\documentclass[11pt]{article}

\usepackage[normalem]{ulem}

\usepackage[english]{babel}
\usepackage[T1]{fontenc}
\usepackage[latin1]{inputenc}

\usepackage{vmargin,graphicx,amsmath,amssymb,amsthm,latexsym,xcolor}
\usepackage{hyperref,fancyhdr,dsfont}

\usepackage{pstricks,pst-grad}%

\usepackage{tikz}
\usetikzlibrary{decorations.pathreplacing}
\usetikzlibrary{decorations.pathmorphing}
\usetikzlibrary{arrows,shapes}
\usetikzlibrary{calc}
\usetikzlibrary{patterns}
\pgfdeclarepatternformonly{mes-hachures-a-moi-personnelles}
{\pgfpoint{-0.1cm}{-0.1cm}}	 
{\pgfpoint{0.3cm}{0.7cm}}	 
{\pgfpoint{0.2cm}{0.6cm}}
{\pgfpathmoveto{\pgfpointorigin}
\pgfpathlineto{\pgfpoint{0.2cm}{0.6cm}}
\pgfusepath{stroke}}
\pgfdeclarepatternformonly{mes-hachures-a-moi-personnelles-2}
{\pgfpoint{-0.1cm}{-0.1cm}}	 
{\pgfpoint{0.3cm}{0.7cm}}	 
{\pgfpoint{0.6cm}{-0.2cm}}
{\pgfpathmoveto{\pgfpointorigin}
\pgfpathlineto{\pgfpoint{0.6cm}{-0.2cm}}
\pgfusepath{stroke}}
\pgfdeclarepatternformonly{mes-hachures-a-moi-personnelles-3}
{\pgfpoint{-0.1cm}{-0.1cm}}	 
{\pgfpoint{0.3cm}{0.7cm}}	 
{\pgfpoint{0.2cm}{0.6cm}}
{\pgfpathmoveto{\pgfpointorigin}
\pgfpathlineto{\pgfpoint{0.2cm}{0.6cm}}
\pgfusepath{stroke}}

\newrgbcolor{white}{1. 1. 1.}
\newrgbcolor{lightblue}{0. 0. 0.80}
\newrgbcolor{lightblue2}{.40 .40 1.}
\newrgbcolor{whiteblue}{.80 .80 1.}
\newrgbcolor{lightred}{.85 0. 0.}
\newrgbcolor{lightred2}{.85 0.40 0.40}
\newrgbcolor{whitered}{0.85 0.45 0.45}
\newrgbcolor{lightgreen}{0. 0.50 0.}
\newrgbcolor{whitegreen}{0.60 0.95 0.40}
\newrgbcolor{lightbrown}{0.90 0.35 0.05}
\newrgbcolor{whitebrown}{0.95 0.55 0.20}

\newcommand{é}{\'e}
\newcommand{è}{\`e}
\newcommand{à}{\`a}
\newcommand{ù}{\`u}
\newcommand{ê}{\^e}
\newcommand{}{\'e}
\newcommand{}{\`e}
\newcommand{}{\`a}
\newcommand{}{\`u}
\newcommand{}{\^e}
\newcommand{\dis}{\displaystyle}
\renewcommand{\{}{\left\lbrace}
\renewcommand{\}}{\right\rbrace}
\renewcommand{\(}{\left(}
\renewcommand{\)}{\right)}
\renewcommand{\[}{\left[}
\renewcommand{\]}{\right]}

\newcommand{\Rin}{\mbox{ \rm in } \,}
\newcommand{\Ron}{\mbox{ \rm on } \,}

\renewcommand{\tilde}{\widetilde}

\newcommand{\ov}{\overset{ \mbox{\tiny def}}{=}}

\newcommand{\Rr}{\mathbb{R}}

\newcommand{\mydiv}{\mathrm{div}\,}

\newcommand{\curl}{\mathrm{curl}\,}

\newcommand{\Gu}{\boldsymbol{u}}

\newcommand{\Gv}{\boldsymbol{v}}
\newcommand{\Gw}{\boldsymbol{w}}
\newcommand{\Gz}{\boldsymbol{z}}
\newcommand{\Gg}{\boldsymbol{g}}
\newcommand{\Gh}{\boldsymbol{h}}

\newcommand{\Gc}{\boldsymbol{c}}

\newcommand{\Gn}{\boldsymbol{\mathrm{n}}}
\newcommand{\Gf}{\boldsymbol{f}}

\newcommand{\GW}{\boldsymbol{W}}

\newcommand{\GG}{\boldsymbol{G}}

\newcommand{\Gpsi}{\boldsymbol{\psi}}

\newcommand{\Gvarphi}{\boldsymbol{\varphi}}

\newcommand{\GV}{\boldsymbol{V}}

\newcommand{\Gy}{\boldsymbol{y}}

\newcommand{\HH}{\mathrm{H}}
\newcommand{\GHH}{\boldsymbol{\mathrm{H}}}

\newcommand{\GWW}{\boldsymbol{\mathrm{W}}}

\newcommand{\LL}{\mathrm{L}}
\newcommand{\GLL}{\boldsymbol{\mathrm{L}}}
\newcommand{\II}{\mathrm{I}}

\newcommand{\dd}{\mathrm{d}}

\newcommand{\nn}{\mathrm{n}}
\newcommand{\0}{\boldsymbol{0}}

\newcommand{\MD}{\mathcal{D}}
\newcommand{\MF}{\mathcal{F}}

\newcommand{\MK}{\mathcal{K}}

\newcommand{\MO}{\mathcal{O}}

\newcommand{\BE}{\begin{equation}}
\newcommand{\EE}{\end{equation}}
\newcommand{\BEn}{\begin{equation*}}
\newcommand{\EEn}{\end{equation*}}
\newcommand{\BEN}{\begin{equation*}}
\newcommand{\EEN}{\end{equation*}}
\newcommand{\BA}{\begin{array}}
\newcommand{\EA}{\end{array}}
\newcommand{\BEQNA}{\begin{eqnarray}}
\newcommand{\EEQNA}{\end{eqnarray}}
\newcommand{\BEQNAn}{\begin{eqnarray*}}
\newcommand{\EEQNAn}{\end{eqnarray*}}
\newcommand{\BEA}{\begin{eqnarray}}
\newcommand{\EEA}{\end{eqnarray}}
\newcommand{\BEAn}{\begin{eqnarray*}}
\newcommand{\EEAn}{\end{eqnarray*}}
\newcommand{\BEAN}{\begin{eqnarray*}}
\newcommand{\EEAN}{\end{eqnarray*}}
\newcommand{\BM}{\begin{multline}}
\newcommand{\EM}{\end{multline}}
\newcommand{\BMn}{\begin{multline*}}
\newcommand{\EMn}{\end{multline*}}
\newcommand{\BMN}{\begin{multline*}}
\newcommand{\EMN}{\end{multline*}}
\newcommand{\vct}[2]{{\left (\begin{array}{c}\displaystyle #1\\\displaystyle #2\end{array}\right )}}
\newcommand{\vctt}[3]{{\left (\begin{array}{c}\displaystyle #1\\\displaystyle #2\\\displaystyle #3\end{array}\right )}}

\newcommand{\fonction}[5]{\begin{array}[t]{lrcl}
#1 :&#2 &\longrightarrow &#3\\
&#4& \longmapsto &#5
\end{array}}

\def \restriction#1#2{\mathchoice
              {\setbox1\hbox{${\displaystyle #1}_{\scriptstyle #2}$}
              \restrictionaux{#1}{#2}}
              {\setbox1\hbox{${\textstyle #1}_{\scriptstyle #2}$}
              \restrictionaux{#1}{#2}}
              {\setbox1\hbox{${\scriptstyle #1}_{\scriptscriptstyle #2}$}
              \restrictionaux{#1}{#2}}
              {\setbox1\hbox{${\scriptscriptstyle #1}_{\scriptscriptstyle #2}$}
              \restrictionaux{#1}{#2}}}
\def\restrictionaux#1#2{{#1\,\smash{\vrule height .8\ht1 depth .85\dp1}}_{\,#2}}

\newcommand{\norm}[1]{\left\Vert #1 \right\Vert}
\newcommand{\abs}[1]{\left\vert #1 \right\vert}
\newcommand{\priv}[2]{#1  \backslash  \overline{#2}}

\reversemarginpar

\hypersetup{
    bookmarks=true,         
    unicode=false,          
    pdftoolbar=true,        
    pdfmenubar=true,        
    pdffitwindow=true,      
    pdftitle={Stability estimates for Navier-Stokes equations and application to inverse problems},    
    pdfauthor={Mehdi Badra, Fabien Caubet, Jrmi Dard},     
    pdfsubject={Stability estimate, Navier-Stokes equations, Carleman inequality, Inverse problems, Unique continuation result},   
    pdfnewwindow=true,      
    pdfkeywords={Stability estimate, Navier-Stokes equations, Carleman inequality, Inverse problems, Unique continuation result}, 
    colorlinks=false,       
    linkcolor=red,          
    citecolor=green,        
    filecolor=magenta,      
    urlcolor=cyan           
}

\newtheorem{theor}{Theorem}[section]
\newtheorem{theorem}[theor]{Theorem}

\newtheorem{prop}[theor]{Proposition}

\newtheorem{lemma}[theor]{Lemma}

\newtheorem{remark}[theor]{Remark}

\title{Stability estimates for Navier-Stokes equations and application to inverse problems}
\author{Mehdi Badra\footnote{Laboratoire LMAP, UMR CNRS 5142, Universit\'e de Pau et des Pays de l'Adour, F-64013 Pau Cedex, France. E-mail: {\tt mehdi.badra@univ-pau.fr}}
\and Fabien Caubet\footnote{Institut de Math\'ematiques de Toulouse ; UMR5219 ; Universit\'e de Toulouse ; CNRS ; UPS IMT, F-31062 Toulouse Cedex 9, France.  E-mail: {\tt fabien.caubet@math.univ-toulouse.fr}
}
\and J\'er\'emi Dard\'e\footnote{Institut de Math\'ematiques de Toulouse ; UMR5219 ; Universit\'e de Toulouse ; CNRS ; UPS IMT, F-31062 Toulouse Cedex 9, France.  E-mail: {\tt jdarde@math.univ-toulouse.fr}}
}

\begin{document}

\maketitle
\cfoot{}
\rhead{\small \textit \leftmark}
\lhead{}
\cfoot{\thepage}
\numberwithin{equation}{section}

\begin{abstract}
 In this work, we  present some new Carleman inequalities for Stokes and Oseen equations with non-homogeneous boundary conditions.  These estimates lead to log type stability inequalities for the problem of recovering the solution of the Stokes and Navier-Stokes equations from both boundary and distributed observations. These inequalities  fit the well-known unique continuation result of Fabre and Lebeau~\cite{FabLeb96}: the distributed observation only depends on interior measurement of the velocity, and the boundary observation only depends on the trace of the velocity and of the Cauchy stress tensor measurements. Finally, we present two applications for such inequalities. First, we apply these estimates to obtain stability inequalities for the inverse problem of recovering Navier or Robin boundary coefficients from boundary measurements. Next, we use these estimates to deduce the rate of convergence of two reconstruction methods of the Stokes solution from the measurement of Cauchy data: a quasi-reversibility method and a penalized Kohn-Vogelius method.
\end{abstract}




\section{Introduction and main results}

For a nonempty bounded open subset $\Omega$ of $\mathbb{R}^N$ ($N=2$ or $N=3$), we 
 consider a pair velocity-pressure $(\Gv,p)\in {\bf H}^2(\Omega) \times \HH^1(\Omega)$ solution of the following linearized Navier-Stokes equations
 (also called Oseen equations):
\begin{equation}  \label{EqVitesseFluideOseen}
\left\lbrace 
\begin{array}{rcll}
- \nu \Delta \Gv + \( \Gz_{1} \cdot \nabla \) \Gv + \( \Gv \cdot \nabla \) \Gz_{2}  + \nabla p    & = & \Gf   & \mbox{in }  \Omega   , \\
 {\rm div}\,  \Gv & = & d   & \mbox{in } \Omega .
\end{array}
\right.
\end{equation}
 Above and in the following, $\nu >0$ is a constant which represents the kinematic viscosity of the fluid, $\Gf \in {\bf L}^2(\Omega)$, $d\in \HH^1(\Omega)$ and 
\begin{equation}\label{Hypzi}
\Gz_{1}\in \GLL^{\infty}(\Omega)\quad \mbox{ and }\quad \Gz_{2}\in \GWW^{1,r} (\Omega)\; \mbox{ with }
\left\lbrace
\begin{array}{rl}
r>2 &\mbox{ if }N=2,\\
r=3 &\mbox{ if } N=3.
\end{array}
\right.
\end{equation}
In the following, $\Gz_{1}$ and $\Gz_{2}$ will be two solutions of the Navier-Stokes equations in $\Omega$. More precisely, if $\Gz_{1}$ and $\Gz_2$ are two solutions of the Navier-Stokes equations, then their difference $\Gv = \Gz_1 - \Gz_2$ verifies \eqref{EqVitesseFluideOseen}.

The pair $(\Gv,p)$ is not completely determined by System~\eqref{EqVitesseFluideOseen}. However, if we have some additional \textit{observation}, such as the value of the velocity $\Gv$ in a nonempty (and arbitrary small) open subset $\omega\subset \Omega$, namely
\BE \label{CondFL01}
\Gv = \Gv_{\rm obs} \quad \mbox{ in } \, \omega ,
\EE
or the value of the Cauchy data ($\Gv$,  $\sigma (\Gv,p)\Gn$) on a  nonempty open subset $\Gamma_{\rm obs}$ of $\partial\Omega$, namely
\begin{equation} \label{CondFL02}
\left\lbrace 
\begin{array}{rcll}
 \Gv   & = & \Gg_{D}   & \mbox{on }  \Gamma_{\rm obs}   , \\
\sigma(\Gv,p) \Gn & = & \Gg_{N}   & \mbox{on } \Gamma_{\rm obs} ,
\end{array}
\right.
\end{equation}
then Fabre and Lebeau's Theorem guarantees the uniqueness of the corresponding pair $(\Gv, p)$ (see \cite{FabLeb96}). However, the related stability inequality expressing the (conditional) continuous dependence of $(\Gv,p)$ with respect to $\|\Gf\|_{{\bf L}^2(\Omega)}$, $\|d\|_{\HH^1(\Omega)}$ and to some norm $\|(\Gv,p)\|_{\rm Obs}$ (corresponding to one of the above mentioned observation) are not yet proved for system \eqref{EqVitesseFluideOseen}.
Indeed, up to our knowledge, the most recent result quantifying the Fabre and Lebeau's unique continuation theorem in the Stokes case is the following one given in~\cite[Theorem~1.4]{BouEgl13-2} by Boulakia~\textit{et al.}:
\begin{theor}[Boulakia~\textit{et al.} in~\cite{BouEgl13-2}]
Assume that $\Omega$ is of class $C^{\infty}$. There exists $d_{0}>0$ such that for all $d>d_{0}$ there exists $C>0$ such that, for all solution $(\Gv,p)\in {\bf H}^2(\Omega)\times \HH^2(\Omega)$ of the Stokes equations
\begin{equation*}
\left\lbrace 
\begin{array}{rcll}
- \nu \Delta \Gv  + \nabla p    & = & \0   & \mbox{\rm in }  \Omega   , \\
 {\rm div}\,  \Gv & = & 0   & \mbox{\rm in } \Omega ,
\end{array}
\right.
\end{equation*}
we have
\begin{equation*}
\norm{\Gv}_{\GHH^1(\Omega)} + \norm{p}_{\HH^1(\Omega)}  \leq   
C \frac{ \|\Gv\|_{{\bf H}^2(\Omega)}+\|p\|_{\HH^2(\Omega)}}{ \left( \ln\left(\displaystyle d \, \frac{\|\Gv\|_{{\bf H}^2(\Omega)}+\|p\|_{\HH^2(\Omega)}}{\|\Gv\|_{\GHH^1(\omega)} + \|p\|_{\HH^1(\omega)} }\right) \right)^{1/2}}
\end{equation*}
and
\begin{equation*}
\norm{\Gv}_{\GHH^1(\Omega)} + \norm{p}_{\HH^1(\Omega)}   \leq  
C  \frac{\|\Gv\|_{{\bf H}^2(\Omega)}+\|p\|_{\HH^2(\Omega)}}{ \left( \ln\left(\displaystyle d \, \frac{\|\Gv\|_{{\bf H}^2(\Omega)}+\|p\|_{\HH^2(\Omega)}}{\|\Gv\|_{{\bf L}^{2}(\Gamma_{\rm obs})} + \|\frac{\partial\Gv}{\partial\Gn}\|_{{\bf L}^{2}(\Gamma_{\rm obs})} + \left\|p\right\|_{\LL^{2}(\Gamma_{\rm obs})}  + \|\frac{\partial p}{\partial\Gn}\|_{{\LL^{2}(\Gamma_{\rm obs})}} }\right) \right)^{1/2}} .
\end{equation*}
\end{theor}
As underlined by the authors themselves, this result does not depend exclusively on the needed observations~\eqref{CondFL01} or~\eqref{CondFL02} and then does not fit the Fabre and Lebeau's Theorem.
The first main results of the present paper are stability inequalities for the Oseen equations~\eqref{EqVitesseFluideOseen} which are quantified versions of Fabre and Lebeau's uniqueness Theorem (see Theorem~\ref{StabilityThmOseen} below) and, in this sense, improve the previous work of Boulakia~\textit{et al.} It allows to obtain analogous stability inequalities for the Navier-Stokes equations. Then, in a second step, we give examples of applications for some parameter identification problems as well as for some error estimates for numerical reconstruction methods.

\paragraph{Stability inequalities.}
In order to state our main theorem, we need some assumptions and notations. Here and in the following, $C>0$ denotes a generic constant which,
unless otherwise stated, only depends on the geometry and which may change from line to line, and~$K \geq e^e$ denotes a constant which satisfies:
\begin{equation}\label{IneqK}
\max \{1\, , \,\norm{\Gz_{1}}_{\GLL^{\infty}(\Omega)} \, , \, \norm{\nabla \Gz_{2}}_{\LL^r(\Omega)} \}
\leq \ln(\ln K).
\end{equation}
Moreover, $\omega$ denotes a nonempty open subset of $\Omega$  and $\Gamma_{\rm obs}$ denotes a nonempty open subset of $\partial\Omega$. In this paper, $\Gn$ is the outward unit normal to $\partial\Omega$ which is assumed to be of class $C^2$ and the stress tensor is defined by $\sigma (\Gu,p) \ov 2\nu  \MD(\Gu) - p \, \II $, where $\II$ is the identity matrix and $\MD(\Gy) \ov \frac{1}{2} \( \nabla \Gy + {}^t \nabla \Gy \)$ is the symmetrized gradient.

We prove (see Subsections~\ref{InagStabdistriObs} and~\ref{StabEstBoundaryObs}) the following 
\begin{theorem} \label{StabilityThmOseen}
Assume \eqref{Hypzi} and \eqref{IneqK} and that $(\Gv,p)\in {\bf H}^2(\Omega)\times \HH^1(\Omega)$ is a solution of the Oseen equations~\eqref{EqVitesseFluideOseen}. For any $M>0$ such that $\|\Gv\|_{{\bf H}^2(\Omega)}+\|p\|_{\HH^1(\Omega)} \leq M$, the following estimates hold:
\begin{equation} \label{ineqStabOseen0}
 \norm{\Gv}_{\GLL^2(\Omega)} \leq   C K \frac{ \, M}{\ln\left(\displaystyle 1+\frac{M}{\|\Gf\|_{\GLL^2(\Omega)} + \norm{d}_{\HH^1(\Omega)} +\|\Gv\|_{\GLL^2(\omega)}}\right)}
\end{equation}
and 
\begin{equation} \label{ineqStabOseen}
\norm{\Gv}_{\GLL^2(\Omega)} \leq   
C K \frac{ M}{\ln\left(\displaystyle 1+\frac{M}{\|\Gf\|_{{\bf L}^{2}(\Omega)} + \norm{d}_{\HH^1(\Omega)} +\|\Gv\|_{{\bf H}^{3/2}(\Gamma_{\rm obs})}+\left\|\sigma(\Gv,p){\bf n}\right\|_{{\bf H}^{1/2}(\Gamma_{\rm obs})}}\right)} .
\end{equation}
Moreover, we have
\BM \label{ineqStabOseen02}
  \norm{{\rm curl} \, \Gv}_{\(\LL^2(\Omega)\)^{2N-3}} +  \norm{ p - {\rm div} \, \Gv}_{\LL^2(\Omega)} \\
  \leq  C K \frac{ M}{ \( \ln\left(\displaystyle 1+\frac{M}{\|\Gf\|_{{\bf L}^{2}(\Omega)} + \norm{d}_{\HH^1(\Omega)} +\|\Gv\|_{{\bf H}^{3/2}(\Gamma_{\rm obs})}+\left\|\sigma(\Gv,p){\bf n}\right\|_{{\bf H}^{1/2}(\Gamma_{\rm obs})}}\right) \)^{1/2}} .
\end{multline}
\end{theorem}

The above theorem allows us to obtain stability estimates for the Navier-Stokes equations. Let $(\Gz_{i},\pi_{i}) \in {\bf H}^2(\Omega)\times \HH^1(\Omega)$, $i=1,2$, satisfy
\begin{equation}\label{EQNS}
\left\lbrace 
\begin{array}{rcll}
- \nu \Delta \Gz_{i} + \( \Gz_{i} \cdot \nabla \) \Gz_{i}  + \nabla \pi_{i}    & = & \Gf   & \mbox{in }   \Omega   , \\
 {\rm div}\,  \Gz_{i} & = & d   & \mbox{in } \Omega .
\end{array}
\right.
\end{equation}
Note that the ${\bf H}^2$ regularity of $\Gz_1$, $\Gz_2$ implies \eqref{Hypzi}. We prove (see Subsection~\ref{sectionStabStokesNS}) the following
\begin{theorem} \label{ThmStabNS}
Suppose that $(\Gz_i,\pi_i)\in {\bf H}^2(\Omega)\times \HH^1(\Omega)$, $i=1,2$, are two solutions of ~\eqref{EQNS} which satisfy \eqref{IneqK} for some $K>e^e$. 
Then, for any $M>0$ such that  $\|\Gz_1-\Gz_2\|_{{\bf H}^2(\Omega)}+\|\pi_{1}-\pi_{2}\|_{\HH^1(\Omega)} \leq M$, the following estimates hold:
\begin{equation} \label{ineqStabNS0}
 \norm{\Gz_1-\Gz_2}_{\GLL^2(\Omega)}
 \leq  C K  \frac{M}{\ln\left(\displaystyle 1+\frac{M}{\|\Gz_1-\Gz_2\|_{\GLL^2(\omega)}}\right)}
\end{equation}
and
\begin{equation}\label{ineqStabNS}
 \norm{\Gz_1-\Gz_2}_{\GLL^2(\Omega)} \leq C K \frac{M}{\ln\left(\displaystyle 1+\frac{M}{\|\Gz_{1} - \Gz_{2}\|_{{\bf H}^{3/2}(\Gamma_{\rm obs})}+\left\|\sigma(\Gz_1,\pi_1){\bf n} - \sigma(\Gz_2,\pi_2){\bf n}\right\|_{{\bf H}^{1/2}(\Gamma_{\rm obs})}}\right)} .
\end{equation}
Moreover, we have
\begin{multline}\label{ineqStabNS02}
 \norm{{\rm curl} \( \Gz_{1} - \Gz_{2} \)}_{\(\LL^2(\Omega)\)^{2N-3}} + 
   \norm{ \pi_1-\pi_2 }_{\LL^2(\Omega)}  \\
 \leq C K \frac{M}{\left( \ln\left(\displaystyle 1+\frac{M}{\|\Gz_{1} - \Gz_{2}\|_{{\bf H}^{3/2}(\Gamma_{\rm obs})}+\left\|\sigma(\Gz_1,\pi_1){\bf n} - \sigma(\Gz_2,\pi_2){\bf n}\right\|_{{\bf H}^{1/2}(\Gamma_{\rm obs})}}\right) \right)^{1/2}} .
\end{multline}
\end{theorem}

We stress  that these stability estimates respect the well known unique continuation result of Fabre and Lebeau (see~\cite{FabLeb96}) since the observation in $\omega$ only concerns the velocity, and since the observation on $\Gamma_{\rm obs}$ only concerns $\restriction{\Gv}{\Gamma_{\rm obs}}$ and $\restriction{\sigma (\Gv,p) \Gn}{\Gamma_{\rm obs}}$. Indeed, Fabre and Lebeau's Theorem states that every velocity $\Gv$ solution of 
\BE \label{StokesIntro}
\{
\BA{rclcl}
- \Delta \Gv + \nabla p & = & \0 & & \Rin \; \Omega , \\
\mydiv \Gv & = & 0 & & \Rin \; \Omega,
\EA
\right.
\EE
which is identically zero in $\omega$ must be zero in $\Omega$ (and then $p$ is constant, see \cite[Proposition~1.1]{FabLeb96} for precise statements). In particular, no information is required on  $p$ to obtain this result. 
Moreover, as a direct consequence of the above mentioned uniqueness result, we can easily deduce that, if a smooth solution $(\Gv,p)$ of System~\eqref{StokesIntro} satisfies $\Gv = \0$ and~$\sigma (\Gv,p) \Gn = \0$ on $\Gamma_{\rm obs}$, then, $\Gv = \0$ and $p=0$ in $\Omega$.
Therefore, inequalities \eqref{ineqStabOseen0},  \eqref{ineqStabOseen} and~\eqref{ineqStabOseen02} are quantifications of Fabre and Lebeau's uniqueness theorem.

The proof of Theorem~\ref{StabilityThmOseen} is  based on  global Carleman inequalities for the Oseen system with non-homogeneous data. 
Quantitative results for unique continuation are classically obtained thanks to Carleman inequalities and three-spheres inequalities. 
We refer to the topical review of Alessandrini \textit{et al.} \cite{Alessandrini_topical_review} and to the references therein for elliptic cases; see also the works of Le Rousseau \textit{et al.} in~\cite{RouLeb12}. However, there is not so much results available on quantitative uniqueness for systems. About Stokes system we shall mention the works of Boulakia \textit{et~al.} in~\cite{BouEgl13, BouEgl13-2} for stability estimates and of Ballerini in~\cite{Ballerini2010} and Lin \textit{et~al.} in~\cite{Uhlmann2010} for some other connected results.

\paragraph{Applications to inverse problems.}

 We obtain stability inequalities for the problem of recovering Navier or Robin boundary coefficients. For this,
we assume that  $\Gamma_{\rm obs}$ and $\Gamma_0$  are two nonempty open subsets of $\partial\Omega$ such that $\Gamma_{\rm obs} \cap \Gamma_{0} = \emptyset $ and
 we consider on $\Gamma_0$ a non penetration condition given by $\Gz \cdot \Gn = 0$ and a friction law given by $2 \nu \[\MD(\Gz){\bf n}\]_{\tau} + \alpha \Gz  =  \0$ (subscript $\tau$ denotes the tangential component). The aim is to reconstruct the friction coefficient~$\alpha$ from Cauchy data on $\Gamma_{\rm obs}$. Thus, we consider two solutions $(\Gz_{i},\pi_{i})\in {\bf H}^2(\Omega)\times \HH^1(\Omega)$ ($i=1,2$) of the Navier-Stokes equations
\begin{equation}\label{EQNS2}
\left\lbrace 
\begin{array}{rcll}
- \nu \Delta \Gz_{i} + \( \Gz_{i} \cdot \nabla \) \Gz_{i}  + \nabla \pi_{i}    & = & \Gf   & \mbox{in }   \Omega   , \\
 {\rm div}\,  \Gz_{i} & = & d   & \mbox{in } \Omega ,
\end{array}
\right.
\end{equation}
associated to two friction coefficients ${\alpha_{i}}\in  \HH^{1/2}(\Gamma_0) \cap \LL^{\infty} (\Gamma_{0})$ ($i=1,2$) 
in the Navier type boundary conditions on $\Gamma_0$:
\begin{equation} \label{CondNavier}
\left\lbrace 
\begin{array}{rcll}
 \Gz_{i} \cdot \Gn & = & 0   & \mbox{on } \Gamma_0,  \\
2 \nu \[\MD(\Gz_{i}){\bf n}\]_{\tau} + \alpha_{i} \Gz_{i} & = & \0   & \mbox{on } \Gamma_0 .
\end{array}
\right.
\end{equation}
We also consider the reconstruction of the Robin coefficient, still denoted $\alpha$, in the case of the classical Robin boundary conditions on $\Gamma_0$ given by:
\begin{equation} \label{CondRobin}
\sigma(\Gz_{i},\pi_{i})\Gn + \alpha_{i} \Gz_{i}  =  \0  \quad  \mbox{on } \Gamma_0 .
\end{equation}
Notice that the $\HH^{1/2}(\Gamma_{0})$-regularity of $\alpha_{i}$ is necessary to have a $\GHH^2(\Omega) \times \HH^1(\Omega)$-regularity of the solutions. 
\begin{theorem} \label{ThmStabReconCoeff}
 Let ${\alpha_{i}}\in \HH^{1/2}(\Gamma_0) \cap \LL^{\infty} (\Gamma_{0})$, $i=1,2$ be two given coefficients. Let $(\Gz_i,\pi_i)\in {\bf H}^2(\Omega)\times \HH^1(\Omega)$, $i=1,2$, be two pairs solution of the Navier-Stokes equations~\eqref{EQNS2} with the boundary conditions~\eqref{CondNavier} or~\eqref{CondRobin} which satisfy \eqref{IneqK} for some $K\geq e^e$. Let $\mathcal{N}\ov \{ x \in \Gamma_0 \, , \, \Gz_{1}(x) = \0\,\mbox{ and }\, \Gz_{2}(x) = \0 \}$, assume that $\MK$ is a compact subset of~$\Gamma_0\backslash \mathcal{N} $ with a nonempty interior and let $m>0$ be a constant such that $\max(\abs{\Gz_{1}},\abs{\Gz_{2}}) \geq m$ on $\MK$. 
Then,  for any $M>0$ such that  $\|\Gz_{1}-\Gz_{2}\|_{{\bf H}^2(\Omega)}+\|\pi_{1}-\pi_{2}\|_{\HH^1(\Omega)} \leq M$, the following inequality holds:
\begin{multline}\label{ineqStabNavier}
 \norm{\alpha_{1} - \alpha_{2}}_{\LL^2(\MK)}  \\
 \leq  \frac{CK}{m} \frac{M}
 { \( \ln\left(\displaystyle 1+\frac{M}{\|\Gz_{1} - \Gz_{2}\|_{{\bf H}^{3/2}(\Gamma_{\rm obs})}+\left\| \sigma(\Gz_{1},\pi_{1})\Gn - \sigma(\Gz_{2},\pi_{2})\Gn\right\|_{{\bf H}^{1/2}(\Gamma_{\rm obs})}}\right) \)^{1/4}} .
\end{multline}
Here, the constant $C$ does not depend only on the geometry but also on $\norm{\alpha_{i}}_{\LL^{\infty}(\Gamma_{0})}$ for $i=1,2$.
\end{theorem}

\begin{remark}
We stress the fact that the previous estimate~\eqref{ineqStabNavier} depends on the solutions~$\Gz_{1}$ and $\Gz_{2}$ through the choice of the compact set $\MK$ and the constant $m$. 
To complete this result, it would be interesting to obtain a quantitative estimate of the vanishing rate of~$\Gz$, like what is done in~\cite{AleSin06} in the case of the Laplace equation.
\end{remark}

\begin{remark}
 Note that the assumptions of Theorem~\ref{ThmStabReconCoeff} guarantee that $\Gz_1$, $\Gz_1$ are continuous. Then if $\mathcal{K}$ exists, the constant $m>0$ exists and depends on $\Gz_1$, $\Gz_1$ on $\MK$. The existence of $\MK$ is known in the case of Robin boundary conditions \eqref{CondRobin} if $\Gz_1$ (or $\Gz_2$) is not identically equal to zero in $\Omega$. It is an easy consequence of Fabre and Lebeau's theorem. But in the case of Navier conditions \eqref{CondNavier} and if one of the $\Gz_i$ is not trivial, the existence of a nonempty  open subset of $\Gamma_0$ on which $\Gz_1$ and $\Gz_2$ both vanish is a difficult issue. Indeed, it reduces to study the existence of a non trivial vector field $\Gv$ solution to an homogeneous Oseen equation (see \eqref{EqStabNavier} below) and such that $\Gv=\partial_{\Gn} \Gv=\0$ on a nonempty open subset of~$\Gamma_0$. The difficulty relies on the fact that, unlike the Robin case, no additional information on the pressure is available.
\end{remark}

\begin{remark}
We can obtain a better estimate assuming more regularity on ~$(\Gv,p)$. More precisely, for $k\geq 2$ and $n\in \mathbb{N}$ suppose that $(\Gv,p) \in \GHH^{k}(\Omega) \times \HH^{k-1}(\Omega)$, $k\geq 2$ and $\alpha_i\in \HH^{n}(\mathcal{K})$, $i=1,2$. Then, using an interpolation argument, we can obtain for any $M>0$ and $N>0$ such that  $ \|\Gv\|_{{\bf H}^2(\Omega)}+\|p\|_{\HH^{1}(\Omega)} \leq M$ and $ \|\Gv\|_{{\bf H}^k(\Omega)}+\|p\|_{\HH^{k-1}(\Omega)} \leq N$ that for all $\theta\in [0,1]$ (see Remark \ref{EstHk}):
\begin{multline}\label{stabInHk}
 \norm{\alpha_{1} - \alpha_{2}}_{\HH^{\theta n}(\mathcal{K})}  \\
  \leq  \frac{
 \left(\frac{CK}{m} N\right)^{1-\theta}\norm{\alpha_{1} - \alpha_{2}}_{\HH^{n}(\mathcal{K})}^\theta}{ \( \ln\left(\displaystyle 1+\frac{M}{\|\Gv_{1} - \Gv_{2}\|_{{\bf H}^{3/2}(\Gamma_{\rm obs})}+\left\| \sigma(\Gv_{1},p_{1})\Gn - \sigma(\Gv_{2},p_{2})\Gn\right\|_{{\bf H}^{1/2}(\Gamma_{\rm obs})}}\right) \)^{\frac{(2k-3)(1-\theta)}{2k}}} .
\end{multline}
For $k=3$ and $\theta=n=0$, we obtain a result similar to the one presented in~\cite[Theorem~4.3]{BouEgl13}. 
\end{remark}
Theorem \ref{ThmStabReconCoeff}, which completes the previous results given by Boulakia \textit{et al} in~\cite{BouEgl13,BouEgl13-2}, finds applications in the modeling of biological problems as blood flow in the cardiovascular system (see~\cite{QuaVen03} and~\cite{VigFig06}) or airflow in the lungs (see~\cite{BafGra10}). 
For the Laplace equation, these kind of stability estimates for the Robin coefficient have been widely studied: see for example the works of Chaabane \textit{et~al.} in~\cite{ChaJao99,ChaFel04}, Alessandrini \textit{et~al.} in~\cite{AleDel03}, Sincich in~\cite{Sin07}, Bellassoued \textit{et~al.} in~\cite{BelChe08} and Cheng \textit{et~al.} in~ \cite{CheCho08}.

Finally, we present another application of our stability estimates in the context of numerical reconstruction methods.  More precisely, we focus on the stable reconstruction of the solution  of a data completion problem (also known as Cauchy problem) for the Stokes equations: for given $(\Gg_D,\Gg_N) \in 
 \GHH^{3/2}(\Gamma_{\rm obs}) \times  \GHH^{1/2}(\Gamma_{\rm obs})$, we search $(\Gv , p ) \in \GHH^2(\Omega) \times \HH^1(\Omega)  $
solution of
\begin{equation} \label{PbStokesErrEstim}
\left\lbrace
\begin{array}{rclcl}
-\nu \Delta \Gv  + \nabla p& =& \Gf  & & \mbox{ in } \, \Omega , \\
 {\rm div}\ \Gv& = &0  & & \mbox{ in } \, \Omega,
 \end{array}
\right.
\end{equation}
and such that
$$ \Gv  = \Gg_D\quad \mbox{ and }\quad \sigma (\Gv,p) \Gn = \Gg_N\quad \mbox{ on }\quad \Gamma_{\rm obs}.$$
Estimates \eqref{ineqStabOseen} and \eqref{ineqStabOseen02}
imply the uniqueness of the solution of the data completion problem.
However, there exists Cauchy data
$(\Gg_D, \Gg_N)$ for which it does not admit any solution. Hence, regularization methods are needed to stably reconstruct $(\Gv,p)$ from
$(\Gg_D,\Gg_N)$. We study two standard regularization methods: a quasi-reversibility regularization and a penalized Kohn-Vogelius regularization.

In the quasi-reversibility method, we consider, for $\varepsilon>0$, the following  variational problem: find $\(\Gv_\varepsilon , p_\varepsilon\) \in \GHH^2(\Omega) \times \HH^1(\Omega)$ such that $\Gv_\varepsilon = \Gg_D$ on~$\Gamma_{\rm obs}$, $\sigma(\Gv_\varepsilon,p_\varepsilon) \Gn =  \Gg_N$ on~$\Gamma_{\rm obs}$, and for all $\( \Gw , q \) \in \GHH^2(\Omega) \times \HH^1(\Omega)$ such that $\Gw = \0 $ and $\sigma(\Gw,q) \Gn = \0$ on $\Gamma_{\rm obs}$, we have 
\begin{multline} \label{MethodeQuasiRevers}
 \int_\Omega  (-\nu \Delta \Gv_\varepsilon + \nabla p_\varepsilon)\cdot (-\nu \Delta \Gw + \nabla q) \, \dd x  
 +  \Big({\rm div}(\Gv_\varepsilon), {\rm div}(\Gw) \Big)_{\HH^1(\Omega)} \\
 + \varepsilon (\Gv_\varepsilon,\Gw)_{\GHH^2(\Omega)} 
+ \varepsilon (p_\varepsilon,q)_{\HH^1(\Omega)} = 
\int_\Omega \Gf  \cdot (-\nu \Delta \Gw + \nabla q)\, \dd x.
\end{multline}

The penalized Kohn-Vogelius approach that we consider here consists in, for $\varepsilon>0$ and $\Gamma_{\rm obs}^C \ov \partial\Omega\backslash\overline{\Gamma_{\rm obs}}$, defining the functional $F_\varepsilon \, : \, \GHH^{1/2}(\Gamma_{\rm obs}^C) \times \GHH^{3/2}(\Gamma_{\rm obs}^C) \rightarrow \Rr$ given by
\begin{multline*}
F_\varepsilon (\Gvarphi_N,\Gpsi_D) \ov
\vert \Gv_{\Gvarphi_N} - \Gv_{\Gpsi_D} \vert_{\GHH^2(\Omega)}^2  + \vert \Gv_{\Gvarphi_N} - \Gv_{\Gpsi_D} \vert_{\GHH^1(\Omega)}^2  \\
+ \varepsilon \Vert \Gv_{\Gvarphi_N},p_{\Gvarphi_N} \Vert_{\GHH^2(\Omega) \times \HH^1(\Omega)}^2 + \varepsilon \Vert \Gv_{\Gpsi_D},p_{\Gpsi_D} \Vert_{\GHH^2(\Omega) \times \HH^1(\Omega)}^2 ,
\end{multline*}
where $\vert . \vert_{\GHH^i(\Omega)}$ is the $\GHH^i$-seminorm ($i=1,2$, see page~\pageref{PbStokesKV} for definition) and where $(\Gv_{\Gvarphi_N},p_{\varphi_{N}}) \in \GHH^2(\Omega) \times \HH^1(\Omega)$ and $(\Gv_{\Gpsi_D},p_{\psi_{D}}) \in \GHH^2(\Omega) \times \HH^1(\Omega)$ are the respective solutions of 
\begin{equation*}
\left\lbrace
\begin{array}{rcll}
\!\!-\nu \Delta \Gv_{\Gvarphi_N} + \nabla p_{\varphi_{N}}& =& \Gf & \! \mbox{ in } \Omega , \\
 {\rm div}\ \Gv_{\Gvarphi_N}& = &0 &\! \mbox{ in } \Omega , \\
 \Gv_{\Gvarphi_N} & = &\Gg_D &\! \mbox{ on } \Gamma_{\rm obs}  , \\
 \sigma(\Gv_{\Gvarphi_N},p_{\varphi}) \Gn &=& \Gvarphi_N &\! \mbox{ on } \Gamma_{\rm obs}^C ,
 \end{array}
\right.
\mbox{and }
\left\lbrace
\begin{array}{rcll}
\!\!-\nu \Delta \Gv_{\Gpsi_D} + \nabla p_{\psi_{D}} & =& \Gf &\! \mbox{ in } \Omega , \\
 {\rm div}\ \Gv_{\Gpsi_D}& = &0 &\! \mbox{ in } \Omega , \\
 \sigma(\Gv_{\Gpsi_D},p_{\psi_{D}}) \Gn &=& \Gg_N  & \!\mbox{ on } \Gamma_{\rm obs}  , \\
 \Gv_{\Gpsi_D} & = &\Gpsi_D  & \! \mbox{ on } \Gamma_{\rm obs}^C .
 \end{array}
\right. 
\end{equation*}
Then, we define $(\Gv_{\varepsilon},p_{\varepsilon}) \ov (\Gv_{\Gvarphi_{N}^\varepsilon},p_{\Gpsi_{D}^\varepsilon})$ where $(\Gvarphi_{N}^\varepsilon,\Gpsi_{D}^\varepsilon) \in \GHH^{1/2}(\Gamma_{\rm obs}^C) \times \GHH^{3/2}(\Gamma_{\rm obs}^C)$ is such that
\begin{equation} \label{eqn_Feps}
F_\varepsilon(\Gvarphi_{N}^\varepsilon,\Gpsi_{D}^\varepsilon)  = \inf_{(\Gvarphi_N,\Gpsi_D) \in \GHH^{1/2}(\Gamma_{\rm obs}^C) \times \GHH^{3/2}(\Gamma_{\rm obs}^C)} F_\varepsilon(\Gvarphi_N,\Gpsi_D) .
\end{equation}
For this second method, we  specify that we assume $\overline{\Gamma_{\rm obs}} \cap \overline{\Gamma_{\rm obs}^C} = \emptyset$,  for instance $\Gamma_{\rm obs} $ could be one of the connected components of $\partial\Omega$.

For any $(\Gg_D,\Gg_N) \in \GHH^{3/2}(\Gamma_{\rm obs}) \times  \GHH^{1/2}(\Gamma_{\rm obs})$, both the quasi-reversibility problem \eqref{MethodeQuasiRevers} and the Kohn-Vogelius minimization problem
\eqref{eqn_Feps} admits a unique solution $(\Gv_\varepsilon, p_\varepsilon)$. Moreover, if the initial data completion problem admits a solution $(\Gv,p)$, then $\Gv_\varepsilon$ converges to $\Gv$ strongly in $\GHH^2(\Omega)$ and $p_\varepsilon$ converges to $p$ strongly in $\HH^1(\Omega)$. Furthermore, the stability estimates we obtain in the present paper (proved in Section~\ref{sectionErrorEsti}) 
provide the rate of convergence of both methods (for a survey on the connection between stability estimates and rates of convergence
of regularization methods, we refer to \cite{Klibanov_review}):
\begin{theorem} \label{ThmRateCv}
For any $M>0$ such that  $ \|\Gv\|_{{\bf H}^2(\Omega)}+\|p\|_{\HH^{1}(\Omega)} \leq M$, where $(\Gv,p)$ is the exact solution of the data completion problem~\eqref{PbStokesErrEstim}, we have the following error estimates for both quasi-reversibility method and penalized Kohn-Vogelius method:
$$
 \Vert \Gv_\varepsilon - \Gv \Vert_{\GLL^2(\Omega)}
 \leq \frac{M}{\ln(1 + \frac{M}{ \sqrt{\varepsilon} })}, \quad 
 \Vert \Gv_\varepsilon - \Gv \Vert_{\GHH^1(\Omega)}
 \leq \frac{M}{\(\ln(1 + \frac{M}{ \sqrt{\varepsilon} })\)^{1/2}}
$$
and
$$
 \Vert p_\varepsilon - p \Vert_{\LL^2(\Omega)} \leq \frac{M}{\(\ln(1 + \frac{M}{ \sqrt{\varepsilon} })\)^{1/2}} .
$$
\end{theorem}

\paragraph{Notations.}
All along this paper, $\Omega$ is a nonempty bounded open subset of $\mathbb{R}^N$ ($N=2$ or $N=3$) with a boundary $\partial\Omega$ of class $C^2$, $\omega$ is a nonempty open subset of $\Omega$, $\Gamma_{\rm obs} $ and~$\Gamma_0$ are nonempty open subsets of~$\partial\Omega$, $\Gamma_{\rm obs} \cap \Gamma_{0} = \emptyset$, and $\Gamma_{\rm obs}^C$ denotes the complement of $\Gamma_{\rm obs}$, namely $\Gamma_{\rm obs}^C \ov \partial\Omega\backslash\overline{\Gamma_{\rm obs}}$.

We here summarize the needed notations in the case $N=3$ which can be easily adapted for $N=2$. 
We denote by $\Gn=\,^t(\nn_1,\nn_2,\nn_3)$ the outward unit normal to $\partial\Omega$ which has a~$C^1$~extension to a neighborhood of $\partial\Omega$. Above and in the following $\,^t$ denotes the transpose. For a scalar function $w$ or a vector field $\Gy=\,^t(y_1,y_2,y_3)$, we define $\nabla w \ov \,^t(\partial_{x_1}w,\partial_{x_2}w,\partial_{x_3}w)$,  $\nabla \Gy\ov (\partial_{x_j} y_i)_{1\leq i,j\leq 3}$ and $\mydiv \Gy \ov \sum_{i=1}^3 \partial_{x_{i}} y_{i}$. Moreover, on~$\partial\Omega$, we define the normal derivatives 
$\frac{\partial w}{\partial \Gn}\ov (\nabla w)\cdot \Gn$ and $\frac{\partial \Gy}{\partial \Gn}\ov (\nabla \Gy) \Gn$ and the tangential gradients $\nabla_\tau w \ov \nabla w- \frac{\partial w}{\partial \Gn} \Gn $ and $\nabla_\tau \Gy \ov  \,^t(\nabla_\tau y_1, \nabla_\tau y_2, \nabla_\tau y_3)$. We also introduce the notations
$\Gy_{\nn}\ov \(\Gy\cdot \Gn\)\Gn$ and
$\Gy_\tau \ov \Gy-\Gy_{\nn} $ for the normal and 
the tangential components of $\Gy$ on $\partial\Omega$. 
The divergence of $\Gy$ is defined by 
${\rm div}\, \Gy\ov \sum_{j=1}^3 \partial_{x_j} y_j$  and the curl of $w$ or $\Gy$ is defined by
$$
{\rm curl}\, \Gy=\partial_{x_1}y_{2}-\partial_{x_2}y_{1}\quad\mbox{ and }\quad {\rm curl}\, w= \vct{\partial_{x_2}w}{-\partial_{x_1}w}\quad\mbox{ if $N=2$},
$$
and
$$
{\rm curl}\, \Gy\ov\vctt{\displaystyle \partial_{x_2}y_{3}-\partial_{x_3}y_{2}}{\displaystyle \partial_{x_3}y_{1}-\partial_{x_1}y_{3}}{\displaystyle \partial_{x_1}y_{2}-\partial_{x_2}y_{1}}\quad\mbox{ if $N=3$}.
$$
We will also need to use the tangential divergence operator on $\partial\Omega$ that we denote by ${\rm div}_\tau$. We recall that $\MD(\Gy) \ov \frac{1}{2} \( \nabla \Gy + {}^t \nabla \Gy \)$ denotes the symmetrized gradient and $\sigma (\Gy,p) \ov 2 \nu  \MD(\Gy) - p \, \II $ the stress tensor, where $\II$ denotes the identity matrix and $\nu >0$ is the constant which represents the kinematic viscosity of the fluid we consider.

For $r\geq 0$ we denote by $\LL^2(\Omega)$, $\LL^2(\partial\Omega)$, $\HH^r(\Omega)$, $\HH^r(\partial\Omega)$, $\HH_0^r(\Omega)$, the usual Lebesgue and Sobolev spaces of scalar functions in $\Omega$ or in $\partial\Omega$, 
and we write in bold the spaces of vector-valued functions: ${\bf L}^2(\Omega)=(\LL^2(\Omega))^N$, ${\bf L}^2(\partial\Omega)=(\LL^2(\partial\Omega))^N$, etc.

We recall that $\Gz_{1}$, $\Gz_{2}$ are vector fields satisfying \eqref{Hypzi}. Moreover, we use the following particular constant:
\begin{equation}\label{DefCz2}
 \mathfrak{m}(\Gz_1,\Gz_2)\ov \max \{1 \, , \,\norm{\Gz_{1}}_{\GLL^{\infty}(\Omega)} \, , \, \norm{\nabla \Gz_{2}}_{\LL^r(\Omega)} \}.
\end{equation}
We also recall that $C>0$ denotes a generic constant only depending on the geometry. In particular, it is independent on $\Gz_1$, $\Gz_2$ and on the parameters $s$, $\lambda$ appearing in Carleman inequalities of sections \ref{sectionCarlemanOseen} and \ref{sectionStabEstimate}.

Finally, for $\mathcal{O}_1$, $\mathcal{O}_2$ two open subsets of $\mathbb{R}^N$, the notation $\mathcal{O}_1 \Subset \mathcal{O}_2$ means that there exists a compact set $\mathcal{K}$  such that $\mathcal{O}_1 \subset \mathcal{K} \subset \mathcal{O}_2$.

\paragraph{Organization of the paper.}
The paper is organized as follows. The Section~\ref{sectionCarlemanOseen} is dedicated to the proof of Carleman inequalities for the non-homogeneous Oseen equations (see Theorem \ref{ThmCarlOseen}). It is obtained by combining a domain extension argument with Carleman inequalities for compactly supported solutions of the Stokes equations. Then in Section~\ref{sectionStabEstimate}, we deduce a H\"older type interior estimates for a distributed observation as well as log type stability inequalities for both distributed and boundary observations. 
In particular, Theorem~\ref{StabilityThmOseen} is proved in subsections~\ref{InagStabdistriObs} and~\ref{StabEstBoundaryObs} and Theorem~\ref{ThmStabNS} is proved in subsection \ref{sectionStabStokesNS}. Finally, we present some applications in the last sections. The Section~\ref{sectionInvPbs} concerns the proof of stability inequalities for the inverse problem of recovering Navier and Robin coefficients (proof of Theorem \ref{ThmStabReconCoeff}) and Section~\ref{sectionErrorEsti} is dedicated to the proof of error estimates for some numerical reconstruction methods (proof of Theorem \ref{ThmRateCv}).


\section{Carleman Inequality for Stokes and Oseen equations}  
\label{sectionCarlemanOseen}

In this section, $\mathcal{O}$ is a non empty bounded open subset of $\mathbb{R}^N$ ($N=2$ or $N=3$) of class~$C^2$, $\omega$ is a non empty bounded open subset such that $\omega\Subset  \mathcal{O}$ and $\psi:{\mathcal{O}}\to \mathbb{R}$ is a function satisfying
\begin{equation}\label{eqpsi}
\begin{split}
\psi\in C^2({\mathcal{O}};\mathbb{R}),\quad &\psi> c_{0}\quad \mbox{ and }\quad |\nabla \psi|>0\quad \mbox{ in }\;\mathcal{O}\backslash \overline{\omega}\\
&\psi=c_0\quad \mbox{ on }\;  \partial \mathcal{O},
\end{split}
\end{equation}
for some positive constant $c_0>0$. For the existence of such a function see for instance \cite{FursikovImanuvilov2006} or \cite[Appendix III]{TucsnakWeiss}. Here, the set $\mathcal{O}$ plays the role of $\Omega$ or of an extension $\Tilde \Omega$ of $\Omega$ which is used in Section \ref{sectionStabEstimate} below.

The main aim of this section is to prove a Carleman inequality for the non homogeneous Oseen equations. For that, we first prove a Carleman inequality for a pair velocity-pressure in ${\bf H}_0^2({\mathcal{O}})\times \HH_0^1({\mathcal{O}})$ and then we use a domain extension argument to recover the non-homogeneous case.


\subsection{Carleman Inequality in the case of homogeneous boundary data}

Let us first recall a standard Carleman inequality for the Laplace equation:

\begin{theorem}\label{CarlEll}
Let $k\in \{0,1\}$, $F\in \LL^2(\mathcal{O})$  and $\GG\in \GLL^2(\mathcal{O})$. 
There exist $C>0$, $\widehat \lambda >1$ and $\widehat s>1$ such that for all $\lambda \geq \widehat \lambda $ and $s\geq \widehat s$, the solution $u\in \HH^1(\mathcal{O})$ of
\BEn
\{
\BA{rclcl}
-\Delta u&= & F+{\rm div\,}\GG & & \Rin \mathcal{O} , \\
u&= &  0 & & \Ron \partial \mathcal{O},
\end{array}
\right.
\EEn
satisfies the following inequality:
\begin{multline}\label{CarlEstEll}
 \int_{\mathcal{O}} \left(e^{(k-1)\lambda \psi}|\nabla u|^2+s^2\lambda ^2 e^{(k+1)\lambda \psi} |u|^2 \right) e^{2 s e^{\lambda \psi}}{\rm d}x  
\\ \leq 
  C\left(\int_{\mathcal{O}}  \left(s e^{k\lambda \psi} |\GG|^2 +s^{-1}\lambda ^{-2}  e^{(k-2)\lambda \psi} |F|^2\right) e^{2 s e^{\lambda \psi}}{\rm d}x \right.\\
  \left. +\int_{\omega} s^2\lambda ^2  e^{(k+1)\lambda \psi} |u|^2e^{2 s e^{\lambda \psi}}{\rm d}x\right).
\end{multline}
\end{theorem}

\begin{proof}
Inequality \eqref{CarlEstEll} for $k=1$ is given for instance in~\cite[Theorem~A.1]{ImmPuel2003} and \eqref{CarlEstEll} for~$k=0$ is obtained by applying 
\eqref{CarlEstEll} with $k=1$ to the equation satisfied by~$e^{-\frac{\lambda}{2} \psi} u$. Note that the above quoted result is stated for a function $\psi$ that vanishes on~$\partial\mathcal{O}$. However, if~$\Tilde s, \Tilde \lambda$ denote the admissible parameters of \cite[Theorem~A.1]{ImmPuel2003}, it suffices to choose $(s,\lambda)=(\Tilde s e^{\Tilde \lambda c_0},\Tilde \lambda)$ to get \eqref{CarlEstEll}.
\end{proof}

We deduce the following Carleman inequality for Stokes equations:

\begin{theorem}\label{CarlLocalStokes}
There exist $C>0$, $\widehat \lambda >1$ and $\widehat s>1$ such that for all $\lambda \geq \widehat \lambda $ and $s\geq \widehat s$, and
for all $(\Gv,p)\in {\bf H}_0^2({\mathcal{O}})\times \HH_0^1({\mathcal{O}})$ the following inequalities hold:
\begin{multline}
 \int_{\mathcal{O}} \left(
|\nabla \Gv|^2 +s  e^{\lambda \psi} |{\rm curl}\, \Gv|^2+s^2\lambda ^2e^{2\lambda \psi} |\Gv|^2\right) e^{2 s e^{\lambda \psi}}{\rm d}x \\
 \leq C \left(\int_{\mathcal{O}} (s^{-1}\lambda^{-2}e^{-\lambda \psi} |\nabla {\rm div}\, \Gv |^2+\lambda ^{-2}|\nabla p -\Delta \Gv|^2) e^{2 s e^{\lambda \psi}}{\rm d}x
 + \int_{\omega}  s^3\lambda ^2e^{3\lambda \psi}|\Gv|^2e^{2 s e^{\lambda \psi}}{\rm d}x \right)
\label{CarlEstStokes1bis}
\end{multline}
and
\BM\label{CarlEstStokesPress}
\int_{\mathcal{O}}  s  e^{\lambda \psi} |{\rm div}\, \Gv-p|^2 e^{2 s e^{\lambda \psi}}{\rm d}x   
\\
 \leq C \left( \int_{\mathcal{O}}  \lambda^{-2}|\nabla p -\Delta \Gv|^2 e^{2 s e^{\lambda \psi}}{\rm d}x + \int_{\omega}  s e^{\lambda \psi} |{\rm div}\, \Gv-p |^2 e^{2 s e^{\lambda \psi}}{\rm d}x\right) .
\end{multline}
\end{theorem}

\begin{proof}
First, we set $\Gf\ov -\Delta \Gv+\nabla p $. Easy calculations yield:
\begin{eqnarray}
-\Delta ({\rm curl}\, \Gv)&=&{\rm curl}\,\Gf \quad \; \, \Rin \mathcal{O},\label{eqPrCarl2}\\
-\Delta ({\rm div}\, \Gv-p )&=&{\rm div}\,\Gf \quad \Rin \mathcal{O},\label{eqPrCarl3}\\
-\Delta \Gv&=&{\rm curl}\,({\rm curl}\, \Gv)-\nabla ({\rm div}\, \Gv)  \quad \Rin \mathcal{O}.\label{eqPrCarl1}
\end{eqnarray}

Then, by applying \eqref{CarlEstEll} for $k=0$ to \eqref{eqPrCarl3} we obtain \eqref{CarlEstStokesPress}. 

Next, we introduce another open subset $\omega_{0} \Subset \omega$ and apply \eqref{CarlEstEll} for $k=0$ to~\eqref{eqPrCarl2} to obtain:
\BM\label{CarlEstStokesRot}
\int_{\mathcal{O}}  s^{-1}\lambda^{-2}  e^{-\lambda \psi} | \nabla \( {\rm curl}\, \Gv \)|^2 e^{2 s e^{\lambda \psi}}{\rm d}x   
 +\int_{\mathcal{O}}  s  e^{\lambda \psi} |{\rm curl}\, \Gv|^2 e^{2 s e^{\lambda \psi}}{\rm d}x   
\\
 \leq C \left(\int_{\omega_0} s e^{\lambda \psi} |{\rm curl}\,\Gv|^2 e^{2 s e^{\lambda \psi}}{\rm d}x
 + \int_{\mathcal{O}}  \lambda^{-2}|\nabla p -\Delta \Gv|^2 e^{2 s e^{\lambda \psi}}{\rm d}x\right).
\end{multline}
Let us replace the local term in $\curl \Gv$ by a local term in~$\Gv$. For that, we introduce a function $\rho\in C^{\infty}_{c} (\omega)$ such that $0\leq\rho\leq 1$ and $\rho=1$ in $\omega_{0}$. Using an integration by parts in~$\omega$, we get
\begin{multline*}
\! \int_{\omega_0} s e^{\lambda \psi} |{\rm curl}\,\Gv|^2 e^{2 s e^{\lambda \psi}}{\rm d}x
 \leq  \, s \int_{\omega} \rho e^{\lambda \psi} |{\rm curl}\,\Gv|^2 e^{2 s e^{\lambda \psi}}{\rm d}x =   \, s \int_{\omega} \curl \! \( \rho e^{\lambda \psi} e^{2 s e^{\lambda \psi}} {\rm curl}\,\Gv \) \Gv \, {\rm d}x \\
 \leq  C \, \bigg( \int_{\omega} s^2\lambda e^{2\lambda \psi} e^{2 s e^{\lambda \psi}} \abs{\Gv}\abs{\curl \Gv}{\rm d}x +  \int_{\omega} se^{\lambda \psi} e^{2 s e^{\lambda \psi}} \abs{\nabla \( \curl \Gv \)} \abs{\Gv}  {\rm d}x  \bigg) ,
\end{multline*}
and with Cauchy-Schwarz inequality:
\begin{multline*}
\int_{\omega_0} s e^{\lambda \psi} |{\rm curl}\,\Gv|^2 e^{2 s e^{\lambda \psi}}{\rm d}x 
\leq  \epsilon \int_{\mathcal{O}}  \left(s^{-1}\lambda^{-2} e^{-\lambda \psi} e^{2 s e^{\lambda \psi}} \abs{\nabla \( \curl \Gv \)}^2 +se^{\lambda \psi}e^{2 s e^{\lambda \psi}} \abs{\curl \Gv}^2\right) {\rm d}x \\
+ \frac{C}{\epsilon}s^3\lambda^2 \int_{\omega} e^{3\lambda \psi} e^{2 s e^{\lambda \psi}} \abs{\Gv}^2 {\rm d}x .
\end{multline*}
By combining~\eqref{CarlEstStokesRot} with the above inequality for $\epsilon>0$ small enough, we obtain 
\BM\label{CarlEstStokesRot2}
\int_{\mathcal{O}}  s^{-1}\lambda^{-2}  e^{-\lambda \psi} | \nabla \( {\rm curl}\, \Gv \)|^2 e^{2 s e^{\lambda \psi}}{\rm d}x   
 +\int_{\mathcal{O}}  s  e^{\lambda \psi} |{\rm curl}\, \Gv|^2 e^{2 s e^{\lambda \psi}}{\rm d}x   
\\
 \leq C \left(\int_{\omega} s^3 \lambda^2 e^{3 \lambda \psi} |\Gv|^2 e^{2 s e^{\lambda \psi}}{\rm d}x+ \int_{\mathcal{O}}  \lambda^{-2}|\nabla p -\Delta \Gv|^2 e^{2 s e^{\lambda \psi}}{\rm d}x\right).
\end{multline}

Finally, \eqref{CarlEstStokes1bis} is obtained by first applying \eqref{CarlEstEll} for $k=1$ to \eqref{eqPrCarl1} and next using the estimate of 
${\rm curl}\, \Gv$ given by \eqref{CarlEstStokesRot2}.
\end{proof}


\subsection{Carleman Inequality in the case of non-homogeneous boundary data}

In this section, we prove a Carleman inequality for the Oseen equations:
\begin{equation}\label{OseenOcal}
\left\lbrace 
\begin{array}{rcll}
- \nu \Delta \Gv + \( \Gz_{1} \cdot \nabla \) \Gv + \( \Gv \cdot \nabla \) \Gz_{2} + \nabla p    & = & \Gf   & \mbox{\rm in }  \MO   , \\
 {\rm div}\,  \Gv & = & d   & \mbox{\rm in } \MO .
\end{array}
\right.
\end{equation}
Above and in the following, $\Gz_{1}\in \GLL^{\infty}(\mathcal{O})$, $\Gz_{2}\in \GWW^{1,r} (\mathcal{O})$ (with $r>2$ if $N=2$ and $r=3$ if $N=3$) and we use the following notation for the particular constant:
\begin{equation}\label{DefCz}
\Tilde {\mathfrak{m}}(\Gz_1,\Gz_2)\ov \max \{1 \, , \,\norm{\Gz_{1}}_{\GLL^{\infty}(\mathcal{O})} \, , \, \norm{\nabla \Gz_{2}}_{\LL^r(\mathcal{O})} \}.
\end{equation}
We recall that $C>0$ denote a generic constant only depending on the geometry and independent on $s$, $\lambda$, $\Gz_1$, $\Gz_2$. 

\begin{theorem} \label{ThmCarlOseen}
There exist $C>0$, $\widehat c>0$ and $\widehat s>1$ such that for all $\Gz_{1}\in \GLL^{\infty}(\mathcal{O})$, $\Gz_{2}\in \GWW^{1,r} (\mathcal{O})$, all $\lambda \geq \widehat \lambda \ov \Tilde {\mathfrak{m}}(\Gz_1,\Gz_2) \widehat c \, $ and all $s\geq \widehat s$ every solution $(\Gv,p)\in {\bf H}^2(\mathcal{O})\times \HH^1(\mathcal{O})$ of \eqref{OseenOcal} satisfies:
\begin{multline}\label{CarlEst3OseenBIS}
 \int_{\MO} \left(|\nabla \Gv|^2 
 +s e^{\lambda \psi} |{\rm curl}\,\Gv|^2 +s^2\lambda ^2e^{2\lambda \psi} |\Gv|^2\right) e^{2 s e^{\lambda \psi}}{\rm d}x\\
 \leq C \left(\int_{\MO} (s^{-1}\lambda^{-2}e^{-\lambda \psi} |\nabla d|^2+\lambda ^{-2}|\Gf|^2) e^{2 s e^{\lambda \psi}}{\rm d}x +\int_{\omega}  s^3\lambda ^2 e^{3\lambda \psi} |\Gv|^2 e^{2 s e^{\lambda \psi}}{\rm d}x \right. \\
\qquad \left. +  e^{2 s e^{\lambda c_0}}  \(\|\Gv\|_{{\bf H}^2(\MO)}^2+\|p\|_{\HH^1(\MO)}^2\)\right)  
\end{multline}
and 
\begin{multline}\label{CarlEst3OseenPressionBIS}
\int_{\mathcal{O}}  s  e^{\lambda \psi} |p-d|^2 e^{2 s e^{\lambda \psi}}{\rm d}x   
\leq C \left( \int_{\omega} ( s^3\lambda ^2 e^{3\lambda \psi} |\Gv|^2+s  e^{\lambda \psi} |p-d|^2  ) e^{2 s e^{\lambda \psi}}{\rm d}x \right. \\
\qquad \left. + \int_{\MO} (s^{-1}\lambda^{-2}e^{-\lambda \psi} |\nabla d|^2+\lambda ^{-2}|\Gf|^2) e^{2 s e^{\lambda \psi}}{\rm d}x
 + e^{2 s e^{\lambda c_0}} \(\|\Gv\|_{{\bf H}^2(\MO)}^2+\|p\|_{\HH^1(\MO)}^2\)\right) .  
\end{multline}
\end{theorem}

\begin{proof}
Let $\Tilde\MO$ be a bounded domain of $\mathbb{R}^N$ of class $C^2$ such that $\MO\Subset  \Tilde\MO$. We extend $\psi$ to~$\Tilde \MO$ (while keeping the same name) in a such a way that:
\begin{equation}
\begin{split}
\psi\in C^2(\Tilde\MO;\mathbb{R}),\quad \psi>0\quad \mbox{ and }\quad |\nabla \psi|>0\quad \mbox{ in }\;\Tilde\MO\backslash \overline{\omega},\label{eqpsi2Oseen}\\
\psi\equiv c_0\quad\mbox{ on } \partial\MO,\quad 0<\psi<c_0\quad\mbox{ in } \Tilde\MO\backslash\overline{\MO}, \quad c_0<\psi\quad\mbox{ in } \MO.
\end{split}
\end{equation}
Let $E \, : \, \GHH^{2}(\MO) \times \HH^{1}(\MO) \rightarrow \GHH^{2}_{0}(\tilde \MO) \times \HH^{1}_{0}(\tilde \MO)$ be a linear continuous map (given for example by Stein's theorem, see~\cite{Adams}), also continuous from $\GHH^{1}(\MO) \times \LL^{2}(\MO)$ into $\GHH^{1}_{0}(\tilde \MO) \times \LL^{2}(\tilde \MO)$, such that $E(\Gv,p)\equiv (\Gv,p)$ in $\MO$, and define $(\Tilde \Gv,\Tilde p) \ov E(\Gv,p)$.  We also denote by $\tilde \Gz_{1}$, $\tilde \Gz_{2}$ some continuous extensions of $\Gz_{1}$, $\Gz_{2}$ for the $\GLL^{\infty}$ and $\GWW^{1,r}$ norms in $\tilde \MO$ respectively. The pair $(\Tilde \Gv,\Tilde p) \ov E(\Gv,p)$ is then solution to
\begin{equation}  \label{EqVitesseFluideTildeOseen}
\left\lbrace 
\begin{array}{rcll} 
- \nu \Delta \Tilde \Gv  + \(\tilde \Gz_{1} \cdot \nabla \) \tilde \Gv  + \(\tilde \Gv \cdot \nabla \) \tilde \Gz_{2} + \nabla \Tilde p    & = & \Tilde \Gf  & \mbox{in }   \Tilde  \MO   , \\
 {\rm div}\,  \Tilde \Gv & = & \Tilde d   & \mbox{in } \Tilde  \MO , \\
\tilde \Gv & = & \0   & \mbox{on }  \partial \Tilde  \MO, \\
\dis \frac{\partial \Tilde \Gv}{\partial {\bf n}}& = & \0   & \mbox{on }  \partial \Tilde  \MO, \\
\tilde p& = & 0   &  \mbox{on }  \partial \Tilde \MO ,
\end{array}
\right.
\end{equation}
where $\Tilde \Gf\in {\bf L}^2(\Tilde\MO)$ and $\Tilde d\in \HH^1(\Tilde\MO)$ are given by $\Tilde \Gf=\Gf$ and $\Tilde d=d$ in $\MO$ and by $ \Tilde \Gf=- \nu \Delta \Tilde \Gv + \(\tilde \Gz_{1} \cdot \nabla \) \tilde \Gv + \(\tilde \Gv \cdot \nabla \) \tilde \Gz_{2}  + \nabla \Tilde p$ and $\Tilde d={\rm div}\,\Tilde \Gv$ in $\Tilde\MO\backslash\overline{\MO}$.
From the continuity of the extension operator $E$ we have:
\begin{equation}
\label{Est14-12-2014-1Oseen}
\|\Tilde \Gf\|_{{\bf L}^2(\Tilde\MO)} + \|\Tilde d\|_{\HH^1(\Tilde\MO)}  \leq C \, \Tilde {\mathfrak{m}}(\Gz_1,\Gz_2) \left(\|\Gv\|_{{\bf H}^2(\MO)}+\|p\|_{\HH^1(\MO)}\right).
\end{equation}
Next, by applying estimate~\eqref{CarlEstStokes1bis} of Theorem~\ref{CarlLocalStokes}:
 \begin{multline} \label{CarlAdaptOseen01}
 \int_{\Tilde\MO}\( \abs{\nabla \Tilde \Gv} 
 + s e^{\lambda \psi} |{\rm curl}\,\Tilde \Gv|^2 +s^2\lambda ^2e^{2\lambda \psi} |\Tilde \Gv|^2 \) e^{2 s e^{\lambda \psi}}{\rm d}x\\
  \leq C \left(\int_{\Tilde\MO} (s^{-1}\lambda^{-2}e^{-\lambda \psi} |\nabla \Tilde d|^2+\lambda ^{-2}|\Tilde \Gf - \(\tilde \Gz_{1} \cdot \nabla \) \tilde \Gv - \(\tilde \Gv \cdot \nabla \) \tilde \Gz_{2} |^2) e^{2 s e^{\lambda \psi}}{\rm d}x\right.\\
  \left. +\int_{\omega}  s^3\lambda ^2 e^{3\lambda \psi} |\Tilde \Gv|^2 e^{2 s e^{\lambda \psi}}{\rm d}x\right).
\end{multline}

Since $\Tilde \Gz_{1}\in \GLL^{\infty} (\Tilde\MO)$, we get
\BE \label{CarlAdaptOseen03}
\int_{\Tilde\MO} \lambda^{-2} \abs{\( \Tilde \Gz_{1} \cdot \nabla\) \Tilde \Gv}^2 e^{2s e^{\lambda \psi}}
\leq  \norm{\Tilde\Gz_{1}}^2_{\GLL^{\infty}(\Tilde\MO) } \lambda^{-2}  \int_{\Tilde\MO}  \abs{\nabla \Tilde \Gv}^2 e^{2s e^{\lambda\psi}} .
\EE
Moreover, since $\Tilde \Gz_{2} \in \GWW^{1,r} (\Tilde \MO)$, we use the H\"older's inequality and the continuous embedding $\GHH^{1}_{0}(\Tilde \MO) \hookrightarrow \GLL^{\frac{2r}{r-2}}(\Tilde\MO)$ to get:
\begin{multline}
\int_{\Tilde\MO} \lambda^{-2} \abs{\( \Tilde \Gv \cdot \nabla\) \Tilde \Gz_{2}}^2 e^{2s e^{\lambda \psi}}{\rm d}x \leq  \int_{\Tilde\MO}  \abs{\nabla \Tilde \Gz_{2}}^2 \abs{\lambda^{-1} \Tilde \Gv e^{s e^{\lambda \psi}}}^2{\rm d}x \\
\begin{array}{rcl}
 & \leq &  \norm{\nabla \Tilde \Gz_{2}}^2_{\LL^r(\Tilde\MO)} \norm{\lambda^{-1} \Tilde \Gv e^{s e^{\lambda \psi}}}^2_{\GLL^{\frac{2r}{r-2}}(\Tilde\MO)} \leq   C \norm{\nabla \Tilde \Gz_{2}}^2_{\LL^r(\Tilde\MO)} \norm{\nabla ( \lambda^{-1} \Tilde \Gv  e^{s  e^{\lambda \psi}} )}^2_{\LL^2 (\Tilde\MO)} \\
 & \leq & \dis C \norm{\nabla \Tilde \Gz_{2}}^2_{\LL^r(\Tilde\MO)} \lambda^{-2} \bigg( \int_{\Tilde\MO} \abs{\nabla \Tilde \Gv}^2 e^{2s e^{\lambda \psi}} {\rm d}x+\int_{\Tilde\MO} \abs{\Tilde \Gv}^2 s^2 \lambda ^2 e^{2\lambda\psi} e^{2s e^{\lambda\psi}} {\rm d}x \bigg) \label{CarlAdaptOseen02}.
 \end{array}
\end{multline}

Thus, gathering~\eqref{CarlAdaptOseen01},~\eqref{CarlAdaptOseen03} and~\eqref{CarlAdaptOseen02} and choosing $\lambda \geq \Tilde {\mathfrak{m}}(\Gz_1,\Gz_2)\, \widehat c $ for $\widehat c$ large enough (and  depending only on the geometry), the terms in $\Tilde \Gz_{1}$, $\Tilde \Gz_{2}$ at the right hand side of inequality~\eqref{CarlAdaptOseen01} can be absorbed and we obtain 
 \begin{multline}\label{EstTildeInterM}
 \int_{\Tilde\MO}\( |\nabla \Tilde \Gv|^2  + s e^{\lambda \psi} |{\rm curl}\,\Tilde \Gv|^2 + s^2\lambda ^2e^{2\lambda \psi} |\Tilde \Gv|^2 \) e^{2 s e^{\lambda \psi}}{\rm d}x\\
 \leq C \left(\int_{\Tilde\MO} (s^{-1}\lambda^{-2}e^{-\lambda \psi} |\nabla \Tilde d|^2+\lambda ^{-2}|\Tilde \Gf|^2) e^{2 s e^{\lambda \psi}}{\rm d}x
 +s^3\lambda^2 \int_{\omega}  e^{3\lambda \psi} |\Tilde \Gv|^2
 e^{2 s e^{\lambda \psi}}{\rm d}x\right).
\end{multline}
Moreover, 
\begin{multline*}
\dis  \int_{\Tilde\MO\backslash \overline{\MO}} (s^{-1}\lambda^{-2}e^{-\lambda \psi} |\nabla \Tilde d|^2+\lambda ^{-2}|\Tilde \Gf|^2) e^{2 s e^{\lambda \psi}}{\rm d}x \\
\leq  \dis \lambda^{-2} e^{2 s e^{\lambda c_0}} \int_{\Tilde\MO\backslash \overline{\MO}} (| \nabla \Tilde d |^2+|\Tilde \Gf|^2){\rm d}x\leq C \lambda^{-2} e^{2 s e^{\lambda c_0}} {\Tilde {\mathfrak{m}}(\Gz_1,\Gz_2)}^2\left(\|\Gv\|_{{\bf H}^2(\MO)}^2+\|p\|_{\HH^1(\MO)}^2\right).
\end{multline*}
In above calculations, we have used the fact that $\psi\leq c_0$ in $\Tilde\MO\backslash \overline{\MO}$ and \eqref{Est14-12-2014-1Oseen}. Then \eqref{CarlEst3OseenBIS} follows by combining the  above inequality with \eqref{EstTildeInterM}.

Finally, to prove~\eqref{CarlEst3OseenPressionBIS}, we first apply~\eqref{CarlEstStokesPress} to $(\Tilde\Gv, \Tilde p)$ which gives:
\BMN
\int_{\Tilde\MO}  s  e^{\lambda \psi} |{\rm div}\, \Tilde \Gv-\Tilde p |^2 e^{2 s e^{\lambda \psi}}{\rm d}x   
\\
 \leq C \left(\int_{\omega}  s e^{\lambda \psi} |{\rm div}\, \Tilde \Gv-\Tilde p |^2 e^{2 s e^{\lambda \psi}}{\rm d}x+\int_{\Tilde\MO}  \lambda^{-2}| \Tilde \Gf - \(\tilde \Gz_{1} \cdot \nabla \) \tilde \Gv - \(\tilde \Gv \cdot \nabla \) \tilde \Gz_{2} |^2 e^{2 s e^{\lambda \psi}}{\rm d}x\right).
\end{multline*}
Then, using \eqref{CarlAdaptOseen03} and \eqref{CarlAdaptOseen02} to estimate the last above integral, we obtain for $\widehat c$ large enough and $\lambda \geq \widehat c \, \Tilde {\mathfrak{m}}(\Gz_1,\Gz_2)$,
\begin{multline*}
\int_{\Tilde\MO}  s  e^{\lambda \psi} |{\rm div}\, \Tilde \Gv-\Tilde p |^2 e^{2 s e^{\lambda \psi}}{\rm d}x  \leq  C \left(\int_{\omega}  s e^{\lambda \psi} |{\rm div}\, \Tilde \Gv-\Tilde p |^2 e^{2 s e^{\lambda \psi}}{\rm d}x +\int_{\Tilde \MO} \lambda ^{-2}|\Tilde \Gf|^2e^{2 s e^{\lambda \psi}}{\rm d}x\right)\\
 + \int_{\Tilde \MO} (|\Tilde \nabla \Gv|^2+s^2\lambda ^2 e^{2\lambda \psi} |\Tilde \Gv|^2
 ) e^{2 s e^{\lambda \psi}}{\rm d}x.
\end{multline*}
Hence, we use~\eqref{EstTildeInterM} and the rest of the proof is the same as for \eqref{CarlEst3OseenBIS}. 
\end{proof}


\section{Stability estimates for Oseen and Navier-Stokes Equations}  \label{sectionStabEstimate}

In this section we use the Carleman inequalities given in Theorem~\ref{ThmCarlOseen} to obtain several stability estimates for both distributed and boundary observation. In particular, we prove Theorems~\ref{StabilityThmOseen} and~\ref{ThmStabNS}.
We first prove a H\"older type interior estimates and a global log type estimates for a distributed observation. Then, we use an extension of the domain procedure to obtain a global log type estimates for a boundary observation.

 We recall that $\Omega$ is a nonempty bounded open subset of $\mathbb{R}^N$ ($N=2$ or $N=3$) with a boundary $\partial\Omega$ of class $C^2$, that $\Gamma_{\rm obs} $ is a nonempty open subset of~$\partial\Omega$ and that $\omega$ is a nonempty open subset of $\Omega$. Moreover, $\Gz_{1}$, $\Gz_{2}$ are vector fields satisfying \eqref{Hypzi} and we use the following notation for the particular constant:
\begin{equation}\label{DefCz2}
\mathfrak{m}(\Gz_1,\Gz_2)\ov \max \{1 \, , \,\norm{\Gz_{1}}_{\GLL^{\infty}(\Omega)} \, , \, \norm{\nabla \Gz_{2}}_{\LL^r(\Omega)} \}.
\end{equation}
We also recall that $C>0$ denotes a generic constant only depending on the geometry and in particular independent on $s$, $\lambda$, $\Gz_1$, $\Gz_2$.


\subsection{Stability estimates with a distributed observation}\label{InagStabdistriObs}


\subsubsection{A H\"older type interior estimate}

\begin{theor} \label{ThmStabOseen01DISTRIBHolder}
Let $\Omega_{0}$ be an open subset such that $\omega \Subset \Omega_{0} \Subset \Omega$. 
There exist $\widehat c>0$, $\widehat s>1$ and $c_{1}^*> c_{2}^*>0$ such that for all $\Gz_{1},\Gz_{2}$ satisfying~\eqref{Hypzi}, all
$\lambda \geq \widehat \lambda \ov \mathfrak{m}(\Gz_1,\Gz_2) \widehat c  $ and all $s\geq \widehat s$, every solution $(\Gv,p)\in {\bf H}^2(\Omega)\times \HH^1(\Omega)$ of the Oseen equations~\eqref{EqVitesseFluideOseen} satisfies:
\BM\label{ineq14-12-2014-3OseenDISTRIBHolder}
\|\Gv\|_{{\bf L}^2(\Omega_{0})} + \| {\rm curl} \, \Gv  \|_{\({\LL}^2(\Omega_{0})\)^{2N-3}}   
\leq 
e^{s  e^{c_{1}^*\lambda}} \left(\|\Gf\|_{{\bf L}^{2}(\Omega)} + \norm{d}_{\HH^1(\Omega)}+\|\Gv\|_{{\bf L}^{2}(\omega)}\right) \\
 +e^{- s  e^{c_{2}^*\lambda}}\left(\|\Gv\|_{{\bf H}^2(\Omega)}+\|p\|_{\HH^1(\Omega)}\right) 
\end{multline}
and
\BM\label{EqThmStab01DISTRIBpHolder}
 \| p - \mydiv \Gv  \|_{{\LL}^2(\Omega_{0})} 
\leq 
e^{s  e^{c_{1}^*\lambda}  } \left(\|\Gf\|_{{\bf L}^{2}(\Omega)} + \norm{d}_{\HH^1(\Omega)}+\|\Gv\|_{{\bf L}^{2}(\omega)} +  \| p \|_{{\LL}^2(\omega)}  \right) \\
 + e^{- s  e^{c_{2}^*\lambda}} \left(\|\Gv\|_{{\bf H}^2(\Omega)}+\|p\|_{\HH^1(\Omega)}\right) .
 \end{multline}
\end{theor}

\begin{proof}
Let us introduce $\dis \psi_{\min} \ov  \min_{x\in \Omega_{0}} \psi (x)$ and $\dis \psi_{\max} \ov \max_{x\in \Omega}\psi(x)$. We apply \eqref{CarlEst3OseenBIS} to $(\Gv,  p)$ to get, with $\widehat \lambda\leq \lambda$,
\begin{multline}  \label{ProofThmStabOseen01}
 \int_{\Omega} \( s^2\lambda ^2e^{2\lambda \psi} | \Gv|^2  +   s e^{\lambda \psi} |{\rm curl}\,  \Gv|^2 
 \) e^{2 s e^{\lambda \psi}} {\rm d}x  \\
\leq 
\dis C \left(\int_{ \Omega}  (| \Gf |^2+| \nabla d |^2) e^{2 s e^{\lambda \psi_{\rm max}}}{\rm d}x \right. 
 + s^3 \lambda^2 e^{3\lambda \psi_{\max}} e^{2 s e^{\lambda \psi_{\max}}} \norm{\Gv}^2_{\GLL^2(\omega)} \\
 \left.  + e^{2 s e^{\lambda c_0}}  \left(  \|\Gv\|_{{\bf H}^2(\Omega)}^2+\|p\|_{\HH^1(\Omega)}^2  \right)\right) 
\end{multline}
and then, 
\begin{multline*}
 \int_{\Omega_{0}} \( s^2\lambda ^2e^{2\lambda \psi_{\min}} | \Gv|^2 + s e^{\lambda \psi_{\min}} |{\rm curl}\,  \Gv|^2  \) e^{2 s e^{\lambda \psi_{\min}}} {\rm d}x 
\leq  
C \left(\int_{ \Omega}  (| \Gf |^2+| \nabla d |^2) e^{2 s e^{\lambda \psi_{\max}}}{\rm d}x \right. \\
 + s^3 \lambda^2 e^{3\lambda \psi_{\max}} e^{2 s e^{\lambda \psi_{\max}}} \norm{\Gv}_{\GLL^2(\omega)}^2 
 \left.  + e^{2 s e^{\lambda c_0}}\left(  \|\Gv\|_{{\bf H}^2(\Omega)}^2+\|p\|_{\HH^1(\Omega)}^2  \right)\right) .  
\end{multline*}
Thus, by dividing the above inequality by $e^{2 s e^{\lambda \psi_{\min}}}$ we obtain \eqref{ineq14-12-2014-3OseenDISTRIBHolder} for some $c_{1}^*> c_{2}^*>0$ independent on $\lambda$. Estimate \eqref{EqThmStab01DISTRIBpHolder} is obtained analogously.
\end{proof}

As a consequence of Theorem \ref{ThmStabOseen01DISTRIBHolder}, we have the following

\begin{theorem}\label{ThmStabOseen01DISTRIBHolder00}
Let $\Omega_{0}$ be an open subset such that $\omega \Subset \Omega_{0} \Subset \Omega$. There exists $c^*>0$ such that for all $\Gz_{1},\Gz_{2}$ satisfying~\eqref{Hypzi}, all $\lambda \geq \widehat \lambda \ov \mathfrak{m}(\Gz_1,\Gz_2) \widehat c$, there exists $\beta\in (0,1/2)$ such that every solution $(\Gv,p)\in {\bf H}^2(\Omega)\times \HH^1(\Omega)$ of the Oseen equations~\eqref{EqVitesseFluideOseen} satisfies:
\begin{multline*}
\norm{\Gv}_{\GLL^2(\Omega_{0})} + \norm{{\rm curl} \, \Gv}_{\(\GLL^2(\Omega_{0})\)^{2N-3}} \\
\leq  e^{e^{c^*\lambda}}\( \|\Gf\|_{{\bf L}^{2}(\Omega)} + \norm{d}_{\HH^1(\Omega)} +\|\Gv\|_{{\bf L}^{2}(\omega)} \)^\beta
\( \|\Gv\|_{{\bf H}^2(\Omega)}+\|p\|_{\HH^1(\Omega)} \)^{1-\beta}
\end{multline*}
and
\begin{multline*}
\norm{ p - {\rm div} \, \Gv}_{\LL^2(\Omega_{0})}\\
\leq    e^{e^{c^*\lambda}} \( \|\Gf\|_{{\bf L}^{2}(\Omega)} + \norm{d}_{\HH^1(\Omega)} +\|\Gv\|_{{\bf L}^{2}(\omega)} +\| p \|_{\LL^{2}(\omega)} \)^\beta
\( \|\Gv\|_{{\bf H}^2(\Omega)}+\|p\|_{\HH^1(\Omega)} \)^{1-\beta}.
\end{multline*}
\end{theorem}
\begin{proof}
Since the proofs are analogous we only prove the first inequality. For that, we apply Theorem~\ref{ThmStabOseen01DISTRIBHolder} and, for $s>\widehat s$, Inequality~\eqref{ineq14-12-2014-3OseenDISTRIBHolder} rewrites $\|\Gv\|_{{\bf L}^2(\Omega_{0})} + \| {\rm curl} \, \Gv  \|_{\({\LL}^2(\Omega_{0})\)^{2N-3}}   \leq e^{s  C_{1}^*} A +e^{- s  C_{2}^*}B$ where $C_{i}^* \ov  e^{c_{i}^*\lambda}$ for $i=1,2$. First, we suppose that $\frac{1}{C_{1}^*+C_{2}^*}\ln\(\frac{B}{A}\)\geq \widehat s$ and we choose $s=\frac{1}{C_{1}^*+C_{2}^*}\ln\(\frac{B}{A}\)$. Hence, we obtain
\BE\label{proofStabHolderDitrib2cas}
\norm{\Gv}_{\GLL^2(\Omega_{0})} + \norm{\curl \Gv}_{\(\GLL^2(\Omega_{0})\)^{2N-3}}
\leq 2 A^{\frac{C^*_{2}}{C^*_{1}+C^*_{2}}} B^{\frac{C^*_{1}}{C^*_{1}+C^*_{2}}} .
\EE
Secondly, if $\frac{1}{C_{1}^*+C_{2}^*}\ln\(\frac{B}{A}\) < \widehat s$, then $B < e^{(C_{1}^*+C_{2}^* )\widehat s} A$. Hence, we also obtain~\eqref{proofStabHolderDitrib2cas} using the existence of $C>0$ such that $\norm{\Gv}_{\GLL^2(\Omega_{0})} + \norm{\curl \Gv}_{\GLL^2(\Omega_{0})} \leq C B$.
\end{proof}
\begin{remark}
According to the proof of Theorem \ref{ThmStabOseen01DISTRIBHolder00}, we have $\beta=\beta(\lambda)=\frac{e^{c_{2}^*\lambda}}{e^{c_{1}^*\lambda}+e^{c_{2}^*\lambda}}$ where $c_1^*>c_2^*>0$ are the constants given in Theorem \ref{ThmStabOseen01DISTRIBHolder} which only depend on the geometry. Therefore, $\beta(\lambda)\in (0,1/2)$ and $\beta(\lambda)\to 0$ as $\lambda\to +\infty$.
\end{remark}


\subsubsection{A global logarithmic estimate}

\begin{theor} \label{ThmStabOseen01DISTRIB}
There exist $\widehat c>0$, $\widehat s>1$ such that for all $\Gz_{1},\Gz_{2}$ satisfying~\eqref{Hypzi}, all $\lambda \geq \widehat \lambda \ov \mathfrak{m}(\Gz_1,\Gz_2) \widehat c$ and all $s\geq \widehat s$, every solution $(\Gv,p)\in {\bf H}^2(\Omega)\times \HH^1(\Omega)$ of the Oseen equations~\eqref{EqVitesseFluideOseen} satisfies:
\BE
\|\Gv\|_{{\bf L}^2(\Omega)} 
\leq 
e^{s  e^{c^*\lambda}  } \left(\|\Gf\|_{{\bf L}^{2}(\Omega)} + \norm{d}_{\HH^1(\Omega)}+\|\Gv\|_{{\bf L}^{2}(\omega)}\right) 
 +\frac{1}{s}\left(\|\Gv\|_{{\bf H}^2(\Omega)}+\|p\|_{\HH^1(\Omega)}\right) ,
 \label{ineq14-12-2014-3OseenDISTRIB} 
\EE
\BE 
\norm{{\rm curl} \, \Gv}_{\(\LL^2(\Omega)\)^{2N-3}}  
\leq 
e^{s  e^{c^*\lambda}  } \left(\|\Gf\|_{{\bf L}^{2}(\Omega)} + \norm{d}_{\HH^1(\Omega)}+\|\Gv\|_{{\bf L}^{2}(\omega)}\right) 
 +\frac{1}{s^{1/2}}\left(\|\Gv\|_{{\bf H}^2(\Omega)}+\|p\|_{\HH^1(\Omega)}\right) 
 \label{EqThmStab01DISTRIB} 
\EE
and
\BM\label{EqThmStab01DISTRIBp}
 \| p - \mydiv \Gv  \|_{{\LL}^2(\Omega)} 
\leq 
e^{s  e^{c^*\lambda}  } \left(\|\Gf\|_{{\bf L}^{2}(\Omega)} + \norm{d}_{\HH^1(\Omega)}+\|\Gv\|_{{\bf L}^{2}(\omega)} +  \| p   \|_{{\LL}^2(\omega)}  \right) \\
 +\frac{1}{s^{1/2}}\left(\|\Gv\|_{{\bf H}^2(\Omega)}+\|p\|_{\HH^1(\Omega)}\right) .
 \end{multline}
\end{theor}

\begin{proof}
We apply \eqref{CarlEst3OseenBIS} to $(\Gv,  p)$, with $\widehat \lambda\leq \lambda$, to get~\eqref{ProofThmStabOseen01} as in Theorem~\ref{ThmStabOseen01DISTRIBHolder}. 
Thus, by dividing inequality~\eqref{ProofThmStabOseen01} by $e^{2 s e^{\lambda c_0}}$ and using the fact that 
$$e^{-2 s e^{\lambda c_0}}\int_{\Omega} e^{2\lambda \psi} |\Gv|^2 e^{2 s e^{\lambda \psi}}{\rm d}x\geq \int_{\Omega}  |\Gv|^2 {\rm d}x$$ 
we obtain \eqref{ineq14-12-2014-3OseenDISTRIB} and \eqref{EqThmStab01DISTRIB}  for some $c^*>0$ large enough (independent on $\lambda$). 
Proceeding as previously (but with~\eqref{CarlEst3OseenPressionBIS} instead of~\eqref{CarlEst3OseenBIS}) we obtain \eqref{EqThmStab01DISTRIBp}.
\end{proof}

Then, we deduce the following

\begin{theorem} \label{StabilityThmOseenDISTRIB}
There exist $\widehat c>0$ and $c^*>0$ such that for all $\Gz_{1},\Gz_{2}$ satisfying~\eqref{Hypzi}, all $\lambda \geq \widehat \lambda \ov \mathfrak{m}(\Gz_1,\Gz_2)\, \widehat c$, every solution $(\Gv,p)\in {\bf H}^2(\Omega)\times \HH^1(\Omega)$ of the Oseen equations~\eqref{EqVitesseFluideOseen} such that $\|\Gv\|_{{\bf H}^2(\Omega)}+\|p\|_{\HH^1(\Omega)} \leq M$ for some $M>0$ satisfies:
\begin{equation}\label{ineqStabOseenDISTRIB}
 \norm{\Gv}_{\GLL^2(\Omega)} \leq   \frac{ { e^{e^{c^* \lambda}}}M}{\ln\left(\displaystyle 1+\frac{M}{\|\Gf\|_{{\bf L}^{2}(\Omega)} + \norm{d}_{\HH^1(\Omega)} +\|\Gv\|_{{\bf L}^{2}(\omega)}}\right)} ,
\end{equation}
\BE\label{ineqStabOseen02DISTRIBcurl}
 \norm{{\rm curl} \, \Gv}_{\(\LL^2(\Omega)\)^{2N-3}} 
 \leq  \frac{{e^{e^{c^* \lambda}}}M}{ \( \ln\left(\displaystyle 1+\frac{M}{\|\Gf\|_{{\bf L}^{2}(\Omega)} + \norm{d}_{\HH^1(\Omega)} +\|\Gv\|_{{\bf L}^{2}(\omega)}}\right) \)^{1/2}}
\EE
and
\BE\label{ineqStabOseen02DISTRIB}
 \norm{ p - {\rm div} \, \Gv}_{\LL^2(\Omega)} 
 \leq  \frac{{e^{e^{c^* \lambda}}}M}{ \( \ln\left(\displaystyle 1+\frac{M}{\|\Gf\|_{{\bf L}^{2}(\Omega)} + \norm{d}_{\HH^1(\Omega)} +\|\Gv\|_{{\bf L}^{2}(\omega)}+\left\|p\right\|_{\LL^{2}(\omega)}}\right) \)^{1/2}} .
\EE
\end{theorem}

\begin{proof}
We apply Theorem~\ref{ThmStabOseen01DISTRIB} and, for $s>\hat s$,  we introduce $A$ such that we rewrite~\eqref{ineq14-12-2014-3OseenDISTRIB} as $ \norm{\Gv}_{\GLL^2(\Omega)} \leq e^{s C^* }A + \frac{c^*}{s} M$ where $C^* \ov  e^{c^* \lambda}$.  First,  if $A=0$, since the previous inequality is true for all $s$, we obtain $\norm{\Gv}_{\GLL^2(\Omega)} =0$ and then \eqref{ineqStabOseenDISTRIB} holds. In the following, we assume $A \neq 0$.

Next, we suppose that $\frac{1}{2C^*}\ln(1+\frac{M}{A})\geq \widehat s$ and we  choose $s=\frac{1}{2C^*}\ln(1+\frac{M}{A})$. It yields
\begin{equation}\nonumber
\begin{split}
\norm{\Gv}_{\GLL^2(\Omega)} \leq   M\left(\left(1+\frac{M}{A}\right)^{1/2}\frac{A}{M}+\frac{2 C^* c^*}{\ln(1+\frac{M}{A}))}\right)
\end{split}
\end{equation}
and next, using the fact that $\frac{1}{x}\leq \frac{1}{\ln(1+x)}$ if $0<x<1$, \textit{i.e.} $M\leq A$ and $\frac{1}{x^{1/2}}\leq \frac{1}{\ln(1+x)}$ if $x>1$, \textit{i.e.} $M>A$, we obtain \eqref{ineqStabOseenDISTRIB} (by choosing $c ^*>0$ larger if necessary). 

In the case $\frac{1}{2C^*}\ln(1+\frac{M}{A})\leq \widehat s$ we have $ M\leq e^{e^{c^* \lambda}} A$ for some  (other) constant $ c^*>0$ and \eqref{ineq14-12-2014-3OseenDISTRIB} with $s=\widehat s$ gives $\norm{\Gv}_{\GLL^2(\Omega)} \leq  e^{e^{c^* \lambda}} A$ for some (other) constant $ c^*>0$. Then the conclusion follows from $A=M\frac{A}{M} \leq M\frac{1}{\ln(1+\frac{M}{A})}$ (since $\frac{1}{x}\leq \frac{1}{\ln(1+x)}$ for all $x>0$).

The proof of~\eqref{ineqStabOseen02DISTRIBcurl} and~\eqref{ineqStabOseen02DISTRIB} are obtained in a similar way.
\end{proof}


\subsection{Stability estimates with boundary observation}\label{StabEstBoundaryObs}

We now prove the following theorem from which we deduce the logarithm estimates stated in Theorem~\ref{StabilityThmOseen} as in the proof of Theorem~\ref{StabilityThmOseenDISTRIB}. Notice that the first estimate~\eqref{ineqStabOseen0} is given in the previous Theorem~\ref{StabilityThmOseenDISTRIB} (see~\eqref{ineqStabOseenDISTRIB}).

\begin{theor} \label{ThmStabOseen01}
There exists $\widehat c>0$, $\widehat s>1$ such that for all $\Gz_{1},\Gz_{2}$ satisfying~\eqref{Hypzi}, all $\lambda \geq \widehat \lambda \ov \mathfrak{m}(\Gz_1,\Gz_2) \widehat c$ and all $s\geq \widehat s$, every solution $(\Gv,p)\in {\bf H}^2(\Omega)\times \HH^1(\Omega)$ of the Oseen equations~\eqref{EqVitesseFluideOseen} satisfies:
\BM\label{ineq14-12-2014-3Oseen}
\|\Gv\|_{{\bf L}^2(\Omega)} \leq 
e^{s  e^{c^*\lambda}  } \left(\|\Gf\|_{{\bf L}^{2}(\Omega)} + \norm{d}_{\HH^1(\Omega)}+\|\Gv\|_{{\bf H}^{3/2}(\Gamma_{\rm obs})}+\left\| \sigma(\Gv,p){\bf n} \right\|_{{\bf H}^{1/2}(\Gamma_{\rm obs})}\right) \\
 +\frac{1}{s}\left(\|\Gv\|_{{\bf H}^2(\Omega)}+\|p\|_{\HH^1(\Omega)}\right) 
\end{multline}
and
\BM\label{EqThmStab01}
\norm{{\rm curl} \, \Gv}_{\(\LL^2(\Omega)\)^{2N-3}}  + \| p - \mydiv \Gv  \|_{{\LL}^2(\Omega)}  \\
\leq 
e^{s  e^{c^*\lambda}  } \left(\|\Gf\|_{{\bf L}^{2}(\Omega)} + \norm{d}_{\HH^1(\Omega)}+\|\Gv\|_{{\bf H}^{3/2}(\Gamma_{\rm obs})}+\left\| \sigma(\Gv,p){\bf n} \right\|_{{\bf H}^{1/2}(\Gamma_{\rm obs})}\right) \\
 +\frac{1}{s^{1/2}}\left(\|\Gv\|_{{\bf H}^2(\Omega)}+\|p\|_{\HH^1(\Omega)}\right) .
\end{multline}
\end{theor}

Let us begin by proving the following lemma which is a construction of an extension of the domain $\Omega$ and of the solution $(\Gv,p)$ of Problem~\eqref{EqVitesseFluideOseen}:

\begin{lemma} \label{lemmaExtension}
Let $\Tilde\Omega$ be an extension of $\Omega$ of class $C^2$ through $\Gamma_{\rm obs}$ (see Figure~\ref{FigNotationDomains}), namely
$$
\Tilde\Omega \mbox{ is of class $C^2$},\quad \partial\Omega \cap \Tilde\Omega =\Gamma_{\rm obs}.
$$
There exists an extension $(\tilde \Gv , \tilde p)\in \GHH^2(\tilde \Omega) \times \HH^1(\tilde \Omega)$ of $(\Gv,p)\in \GHH^2(\Omega) \times \HH^1(\Omega)$ such that
\BEAN
\restriction{\tilde \Gv}{\Gamma_{\rm obs}} = \restriction{\Gv}{\Gamma_{\rm obs}} ,
& \restriction{\frac{\partial \tilde \Gv}{\partial\Gn}}{\Gamma_{\rm obs}} = \restriction{\frac{\partial\Gv}{\partial\Gn}}{\Gamma_{\rm obs}}  ,
& 
\restriction{\tilde p}{\Gamma_{\rm obs}} = \restriction{p}{\Gamma_{\rm obs}} 
\EEAN
with the following estimate
\begin{equation}\label{estExt4TraceOseen}
\|\Tilde \Gv\|_{{\bf H}^2(\priv{\Tilde\Omega}{\Omega})}^2+\|\Tilde p\|_{\HH^1(\priv{\Tilde\Omega}{\Omega})}^2\leq C\left(\|\Gv\|_{{\bf H}^{3/2}(\Gamma_{\rm obs})}^2 + \norm{\frac{\partial\Gv}{\partial\Gn}}_{{\bf H}^{1/2}(\Gamma_{\rm obs})}^2 + \|p\|_{{\HH}^{1/2}(\Gamma_{\rm obs})}^2\right) .
\end{equation}
In particular,
\begin{equation}\label{estExt4Oseen}
\|\Tilde \Gv\|_{{\bf H}^2(\Tilde\Omega)}^2+\|\Tilde p\|_{\HH^1(\Tilde\Omega)}^2\leq C\left(\|\Gv\|_{{\bf H}^2(\Omega)}^2+\|p\|_{\HH^1(\Omega)}^2\right).
\end{equation}
\end{lemma}

\begin{proof}
We consider the linear continuous trace-right inverse operator (see for example~\cite[Theorem~1.5.1.2]{Grisvard})
$$
\fonction{R}
{{\bf H}^{3/2}(\Gamma_{\rm obs})\times {\bf H}^{1/2}(\Gamma_{\rm obs})\times \HH^{1/2}(\Gamma_{\rm obs})}
{{\bf H}^2(\Omega)\times \HH^1(\Omega)}
{(\Gg_{\rm obs},\Gh_{\rm obs}, k_{\rm obs})}
{(\Gw, q)}
$$
with $(\Gw,\frac{\partial \Gw}{\partial {\bf n}},q)=(\Gg_{\rm obs},\Gh_{\rm obs}, k_{\rm obs} )$ on $\Gamma_{\rm obs}$. Then, let us denote by $S$ the linear continuous extension operator given by Stein's theorem (see~\cite{Adams}):
$$
\fonction{S}
{{\bf H}^{2}(\Omega)\times \HH^{1}(\Omega)}
{{\bf H}^2(\Rr^N)\times \HH^1(\Rr^N)}
{(\Gw,q)}
{(\GW, Q)} .
$$
We also denote by $T$ the linear continuous operator of restriction to $\tilde \Omega$:
$$
\fonction{T}
{{\bf H}^{2}(\Rr^N)\times \HH^{1}(\Rr^N)}
{{\bf H}^2(\tilde \Omega)\times \HH^1(\tilde \Omega)}
{(\GW,Q)}
{(\restriction{\GW}{\tilde \Omega}, \restriction{Q}{\tilde \Omega})} .
$$
Finally, by denoting $(\tilde \Gw, \tilde q) \ov T \circ S \circ R \, (\Gv,\partial_{\nn}\Gv,p) $, we conclude by defining
$$
\tilde \Gv \ov
\{ \begin{array}{rl}
\Gv & \mbox{ in } \Omega \\
\tilde \Gw & \mbox{ in } \priv{\tilde \Omega}{\Omega}
\end{array} \right.
\qquad \mbox{ and } \qquad
\qquad
\tilde q \ov
\{ \begin{array}{rl}
p & \mbox{ in } \Omega \\
\tilde q & \mbox{ in } \priv{\tilde \Omega}{\Omega} .
\end{array} \right.
$$
It is easily checked that $(\tilde \Gv, \tilde p) \in \GHH^2(\tilde \Omega) \times \HH^1(\tilde \Omega)$.
\end{proof}

\begin{proof}[Proof of Theorem~\ref{ThmStabOseen01}]
In what follows, $\tilde \Gz_{1}$, $\tilde \Gz_{2}$ denote some continuous extensions to $\mathbb{R}^N$ of $\Gz_{1}$, $\Gz_{2}$, for the $\GLL^{\infty}$ and the $\GWW^{1,r}$  norm respectively.
Let us consider the extensions $\tilde \Omega$ and $(\tilde \Gv, \tilde q) \in \GHH^2(\tilde \Omega) \times \HH^1(\tilde \Omega)$ given by Lemma~\ref{lemmaExtension}. 
Let us consider $\omega\Subset  \Tilde\Omega\backslash\overline{\Omega}$ a non empty bounded open subset. 
We summarize these notations in Figure~\ref{FigNotationDomains}. 
\begin{figure}[h] 
\centering
\begin{tikzpicture}[scale=2.0]
%
\fill[whiteblue] (2,1) .. controls (2.5,1) and (3,1.3) .. (3.4,0.7) 
		  .. controls (3.7,0.3) and (3.3,-0.2) .. (2.5,-0.6) ;
\fill[whiteblue] (0,0) .. controls (1,1) and (1.9,1) .. (2,1) 
		  .. controls (2.5,1) and (3.1,-0.3) .. (2.5,-0.6)
		  .. controls (2.35,-0.7) and (1,-1.0) .. (0.5,-1)
		  .. controls (0,-1) and (-1,-1) .. (0,0) ;
\coordinate [label={[lightblue2]right:$\Tilde\Omega$}] (OmegaTilde) at (-0.15,0.25);
%
\draw (0,0) .. controls (1,1) and (1.9,1) .. (2,1) 
		  .. controls (2.5,1) and (3.1,-0.3) .. (2.5,-0.6) 
		  .. controls (2.35,-0.7) and (1,-1.0) .. (0.5,-1) 
		  .. controls (0,-1) and (-1,-1) .. (0,0);
\coordinate [label=right:$\Omega$] (Omega) at (0.5,-0.7);
%
\draw (2,1) .. controls (2.5,1) and (3,1.3) .. (3.4,0.7) 
		  .. controls (3.7,0.3) and (3.3,-0.2) .. (2.5,-0.6) ;
%
\draw[lightbrown,line width=1.1pt] (2.97,0.5) circle (0.15cm) node {$\boldsymbol{\omega}$ } ;
%
\draw (2,1) .. controls (2.5,1) and (3.1,-0.3) .. (2.5,-0.6) ;
\draw[red,line width=1.1pt,dotted] (2,1) .. controls (2.5,1) and (3.1,-0.3) .. (2.5,-0.6) ;
\coordinate [label={[red]right:$\Gamma_{\rm obs}$}] (Gammaobs) at (2.1,0.4);
\end{tikzpicture}
\caption{Notations}\label{FigNotationDomains}
\end{figure}

Next, we apply \eqref{CarlEst3OseenBIS} and \eqref{CarlEst3OseenPressionBIS} to $(\Tilde\Gv, \Tilde p)$ and, with $\widehat\lambda\leq \lambda$, we get:
\begin{multline*}
 \int_{\Tilde\Omega} \( s^2\lambda ^2e^{2\lambda \psi} |\Tilde \Gv|^2 
  +s e^{\lambda \psi} |{\rm curl}\, \Tilde \Gv|^2 
  + s  e^{\lambda \psi} | \Tilde  p- \mydiv \Tilde \Gv |^2 
 \) e^{2 s e^{\lambda \psi}} {\rm d}x \\
\leq C \left(\int_{\Tilde \Omega} (|- \nu \Delta \Tilde \Gv + \( \Tilde \Gz_{1} \cdot \nabla \) \Tilde \Gv   + \( \Tilde \Gv \cdot \nabla \) \Tilde \Gz_{2}  + \nabla \Tilde p   |^2+| \nabla {\rm div}\,\Tilde \Gv|^2) e^{2 s e^{\lambda \psi}}{\rm d}x \right. \\
\left. +\int_{\omega} (  s^3\lambda ^2 e^{3\lambda \psi} |\Tilde \Gv|^2 
   + s  e^{\lambda \psi} | \Tilde  p- \mydiv \Tilde \Gv |^2 
  ) e^{2 s e^{\lambda \psi}}{\rm d}x 
  + e^{2 s e^{\lambda c_0}} \left(\|\Gv\|_{{\bf H}^2(\Tilde\Omega)}^2+\| p \|_{\HH^1(\Tilde \Omega)}^2\right)\right),
\end{multline*}
and from \eqref{estExt4TraceOseen}, \eqref{estExt4Oseen} and  $\omega\subset  \Tilde\Omega\backslash \overline{\Omega}$ we deduce 
\begin{multline*}
 \int_{\Omega} \( s^2\lambda ^2e^{2\lambda \psi} | \Gv|^2 
  +s e^{\lambda \psi} |{\rm curl}\,  \Gv|^2 
  + s  e^{\lambda \psi} |   p- \mydiv  \Gv |^2 
 \) e^{2 s e^{\lambda \psi}} {\rm d}x \\
\begin{array}{rcl}
  &\leq & \dis C \Bigg(\int_{\Omega}  \(  \abs{ \nabla d}^2 +  |\Gf|^2 \) e^{2 s e^{\lambda \psi}}{\rm d}x
  +e^{3\lambda \psi_{\max}} e^{2 s e^{\lambda \psi_{\max}}} \int_{\Tilde\Omega\backslash \overline{\Omega}}  \bigg(s^3\lambda ^2|\Tilde \Gv|^2  \\
& & \qquad    + s  | \Tilde  p- \mydiv \Tilde \Gv |^2   
 +|\Delta \Tilde \Gv|^2  +  \abs{\( \Tilde \Gz_{1} \cdot \nabla \) \Tilde \Gv}^2 + \abs{\( \Tilde \Gv \cdot \nabla \) \Tilde \Gz_{2}}^2 + |\nabla \Tilde p|^2 + |\nabla {\rm div}\,\Tilde \Gv|^2 \bigg) {\rm d}x \\
& & \qquad +e^{2 s e^{\lambda c_{0} }}  \left(\|\Gv\|_{{\bf H}^2(\Omega)}^2+\|p\|_{\HH^1(\Omega)}^2\right)\Bigg)\\
 &\leq & \dis C \left(s^3\lambda^2 e^{3\lambda \psi_{\max}} e^{2 s e^{\lambda \psi_{\max}}} \mathfrak{m}(\Gz_1,\Gz_2)^2 \left(\|\Gv\|_{{\bf H}^{3/2}(\Gamma_{\rm obs})}^2+\left\|\frac{\partial \Gv}{\partial {\bf n}}\right\|_{{\bf H}^{1/2}(\Gamma_{\rm obs})}^2  + \|p\|_{{\bf H}^{1/2}(\Gamma_{\rm obs})}^2\right )\right. \\ 
& & \dis \left.\qquad \int_{\Omega}  \(  \abs{\nabla d}^2 + |\Gf|^2 \) e^{2 s e^{\lambda \psi}}{\rm d}x
 +e^{2 s e^{\lambda c_0}} \left(\|\Gv\|_{{\bf H}^2(\Omega)}^2+\|p\|_{\HH^1(\Omega)}^2\right)\right).
\end{array}
\end{multline*}
Here, we have used the notation $\psi_{\max}\ov \max_{x\in  \Tilde\Omega}\psi(x)$.
Thus, by dividing the above inequality by $e^{2 s e^{\lambda c_0}}$ and using  that 
$$e^{-2 s e^{\lambda c_0}}\int_{\Omega} e^{2\lambda \psi} |\Gv|^2 e^{2 s e^{\lambda \psi}}{\rm d}x\geq \int_{\Omega}  |\Gv|^2 {\rm d}x$$ 
and that $\lambda \geq \mathfrak{m}(\Gz_1,\Gz_2) \widehat c$, we obtain for some $c^*>0$ large enough (independent on $\lambda$),
\BM\label{ineq14-12-2014-1Oseen}
\|\Gv\|_{{\bf L}^2(\Omega)}^2\leq 
e^{s  e^{c^*\lambda}  } \left(\|\Gf\|_{{\bf L}^{2}(\Omega)}^2 + \norm{d}^2_{\HH^1(\Omega)}+\|\Gv\|_{{\bf H}^{3/2}(\Gamma_{\rm obs})}^2+\left\|\frac{\partial \Gv}{\partial {\bf n}}\right\|_{{\bf H}^{1/2}(\Gamma_{\rm obs})}^2+\|p\|_{{\HH}^{1/2}(\Gamma_{\rm obs})}^2\right) \\
 +\frac{1}{s^2}\left(\|\Gv\|_{{\bf H}^2(\Omega)}^2+\|p\|_{\HH^1(\Omega)}^2\right) .
\end{multline}
With a similar argument, 
\BM\label{ineq14-12-2014-1OseenRotPression}
\norm{{\rm curl} \, \Gv}_{\(\LL^2(\Omega)\)^{2N-3}}^2 + \| p - \mydiv \Gv \|_{\LL^2(\Omega)}^2 \leq 
e^{ s e^{c^* \lambda} } \left(\|\Gf\|_{{\bf L}^{2}(\Omega)}^2 + \norm{d}^2_{\HH^1(\Omega)}+\|\Gv\|_{{\bf H}^{3/2}(\Gamma_{\rm obs})}^2 \right. \\
  \left. +\left\|\frac{\partial \Gv}{\partial {\bf n}}\right\|_{{\bf H}^{1/2}(\Gamma_{\rm obs})}^2+\|p\|_{{\HH}^{1/2}(\Gamma_{\rm obs})}^2\right) 
 +\frac{1}{s }\left(\|\Gv\|_{{\bf H}^2(\Omega)}^2+\|p\|_{\HH^1(\Omega)}^2\right) .
\end{multline}

Now, to conclude, it remains to replace the term $\left\|\frac{\partial \Gv}{\partial {\bf n}}\right\|_{{\bf H}^{1/2}(\Gamma_{\rm obs})}^2+\|p\|_{{\HH}^{1/2}(\Gamma_{\rm obs})}^2$ by $\left\|\sigma(\Gv,p){\bf n}\right\|_{{\bf H}^{1/2}(\Gamma_{\rm obs})}^2$. First, from 
$$
d={\rm div}\, \Gv=\frac{\partial \Gv}{\partial {\bf n}}\cdot {\bf n} +{\rm div}_\tau \,\Gv_\tau +({\rm div}\,{\bf n}) (\Gv \cdot \Gn)\quad \mbox{ on }\quad \Gamma_{\rm obs} ,
$$
we deduce that 
\begin{equation}\label{ineq14-12-2014-2Oseen}
\left\|\frac{\partial \Gv}{\partial {\bf n}}\cdot {\bf n}\right\|_{\HH^{1/2}(\Gamma_{\rm obs})}^2 
\leq C \( \left\|\Gv\right\|_{{\bf H}^{3/2}(\Gamma_{\rm obs})}^2 + \norm{d}^2_{\HH^1(\Omega)} \).
\end{equation}
The above inequality with the following computations
\begin{eqnarray*}
\nu \frac{\partial \Gv}{\partial {\bf n}}&= & \sigma(\Gv,p){\bf n}+ p{\bf n}-\nu\, ^t\nabla \Gv \Gn\\
&= & \sigma(\Gv,p){\bf n}+ p{\bf n}-\nu \nabla (\Gv\cdot {\bf n}) + \nu(\nabla {\bf n}) \,  \Gv\\
&= & \sigma(\Gv,p){\bf n}+ p{\bf n}-\nu\(\frac{\partial \Gv}{\partial {\bf n}}\cdot {\bf n}\) \Gn-\nu \nabla_\tau (\Gv\cdot {\bf n}) + \nu(\nabla {\bf n}) \, \Gv,
\end{eqnarray*}
yields
$$
\left\|\frac{\partial \Gv}{\partial {\bf n}}\right\|_{{\bf H}^{1/2}(\Gamma_{\rm obs})}^2
\leq C\left(\left\|\sigma(\Gv,p){\bf n}\right\|_{{\bf H}^{1/2}(\Gamma_{\rm obs})}^2 + \left\|p\right\|_{\HH^{1/2}(\Gamma_{\rm obs})}^2 + \left\|\Gv\right\|_{{\bf H}^{3/2}(\Gamma_{\rm obs})}^2 + \norm{d}^2_{\HH^1(\Omega)} \right) .
$$
Finally, from $p= 2 \nu \frac{\partial \Gv}{\partial {\bf n}}\cdot {\bf n}-\sigma(\Gv,p){\bf n}\cdot {\bf n}$ and \eqref{ineq14-12-2014-2Oseen} we deduce that
$$
\left\|p\right\|_{\HH^{1/2}(\Gamma_{\rm obs})}^2
\leq C\left(\left\|\sigma(\Gv,p){\bf n}\right\|_{{\bf H}^{1/2}(\Gamma_{\rm obs})}^2 + \left\|\Gv\right\|_{{\bf H}^{3/2}(\Gamma_{\rm obs})}^2 + \norm{d}^2_{\HH^1(\Omega)} \right)
$$
and then
\BE \label{DerniereEquationProofStab}
\left\|\frac{\partial \Gv}{\partial {\bf n}}\right\|_{{\bf H}^{1/2}(\Gamma_{\rm obs})}^2 + \left\|p\right\|_{\HH^{1/2}(\Gamma_{\rm obs})}^2
\leq C\left(\left\|\sigma(\Gv,p){\bf n}\right\|_{{\bf H}^{1/2}(\Gamma_{\rm obs})}^2 + \left\|\Gv\right\|_{{\bf H}^{3/2}(\Gamma_{\rm obs})}^2 + \norm{d}^2_{\HH^1(\Omega)} \right) .
\EE
Then, \eqref{ineq14-12-2014-3Oseen} and \eqref{EqThmStab01} follow from~\eqref{ineq14-12-2014-1Oseen}, \eqref{ineq14-12-2014-1OseenRotPression} and \eqref{DerniereEquationProofStab}.
\end{proof}


\subsection{ Proof of the stability estimates for the Navier-Stokes equations} \label{sectionStabStokesNS}

Theorem~\ref{ThmStabNS} is a simple consequence of Theorem~\ref{StabilityThmOseen} applied to the pair~$(\Gv \, , \, p)\ov (\Gz_{1}-\Gz_{2} \, , \,  \pi_{1}-\pi_{2})$ which is solution of:
\begin{equation*}
\left\lbrace 
\begin{array}{rcll}
- \nu \Delta \Gv +  \( \Gz_{1} \cdot \nabla \) \Gv + \( \Gv \cdot \nabla \) \Gz_{2}    + \nabla p    & = & \0   & \mbox{ in }   \Omega   , \\
 {\rm div}\,  \Gv & = & 0   & \mbox{ in } \Omega , \\
\Gv & = & \Gz_{1} - \Gz_{2}   & \mbox{ on } \Gamma_{\rm obs}, \\
\sigma(\Gv,p){\bf n} & = & \sigma(\Gz_1,\pi_1){\bf n} - \sigma(\Gz_2,\pi_2){\bf n}   & \mbox{ on } \Gamma_{\rm obs} .
\end{array}
\right.
\end{equation*}

Note that in the same way, we can also obtain the same estimates as in Theorem \ref{ThmStabOseen01DISTRIBHolder00} and Theorem \ref{StabilityThmOseenDISTRIB} for a distributed observation.


\section{Application: stability estimates for boundary coefficients inverse problems} \label{sectionInvPbs}

In the present section, we focus on the proof of Theorem~\ref{ThmStabReconCoeff}. 
We begin by considering the Navier boundary conditions. One can first notice that the pair $(\Gv \, , \, p) \ov (\Gz_{1}-\Gz_{2} \, , \,  \pi_{1}-\pi_{2})$ satisfies
\begin{equation}  \label{EqStabNavier}
\left\lbrace 
\begin{array}{rcll}
- \nu \Delta \Gv +   \( \Gz_{1} \cdot \nabla \) \Gv  +   \( \Gv \cdot \nabla \) \Gz_{2}  + \nabla p    & = & \0   & \mbox{in }   \Omega  ,  \\
 {\rm div}\,  \Gv & = & 0   & \mbox{in } \Omega , \\
 \Gv \cdot \Gn & = & 0   & \mbox{on } \Gamma_0 ,\\
2 \nu \[\MD(\Gv){\bf n}\]_{\tau} + \alpha_{1} \Gz_{1} - \alpha_{2} \Gz_{2} & = & \0   & \mbox{on } \Gamma_0 .
\end{array}
\right.
\end{equation}
Without loss of generality, we can assume that $\abs{\Gz_{1}} \geq m$ on $\MK$. Then, since $(\alpha_{2} - \alpha_{1}) \Gz_{1} = \alpha_{2} \Gv + 2 \nu \[ \MD(\Gv) \Gn \]_{\tau}$ on $\Gamma_{0}$, we have
\begin{equation}\label{est1Diffalpai}
\norm{\alpha_{1} - \alpha_{2}}_{\LL^2(\MK)}   
\leq \frac{C}{m} \( \norm{\Gv}_{\GLL^2(\Gamma_0)} + \norm{\nabla \Gv}_{\LL^2(\Gamma_0)} \) .
\end{equation}
To estimate the above right hand side, we use the following inequalities:
\begin{equation}\label{IneqInterp}
\norm{\Gv}_{\GLL^2(\Gamma_0)} \leq C \norm{\Gv}^{1/2}_{\GLL^2(\Omega)} \norm{\Gv}^{1/2}_{\GHH^1(\Omega)}
\quad \mbox{ and } \quad
\norm{\nabla \Gv}_{\LL^2(\Gamma_0)} \leq C \norm{\Gv}^{1/2}_{\GHH^1(\Omega)} \norm{\Gv}^{1/2}_{\GHH^2(\Omega)}.
\end{equation}
Note that the above inequalities are an immediate consequence of the interpolation inequality $\norm{\cdot}^2_{\LL^2(\partial\Omega)} \leq C \norm{\cdot}_{\HH^1(\Omega)}\norm{\cdot}_{\LL^2(\Omega)}$ which can be obtained for instance by first applying ~\cite[Theorem~1.5.1.10]{Grisvard} to get $C>0$ such that for all $u\in \HH^1(\partial\Omega)$ and all $0<\varepsilon <1$,
$$\norm{u}^2_{\LL^2(\partial\Omega)} \leq C \( \varepsilon^{1/2} \norm{\nabla u}^2_{\GLL^2(\Omega)} + \varepsilon^{-1/2} \norm{u}^2_{\LL^2(\Omega)}  \),$$
and next by taking 
$\dis \varepsilon = \norm{u}^2_{\LL^2(\Omega)}/\norm{u}^2_{\HH^1(\Omega)}$.
Combining the interpolation inequality $\norm{\Gv}_{\GHH^1(\Omega)} \leq \norm{\Gv}_{\GLL^2(\Omega)}^{1/2} \norm{\Gv}_{\GHH^2(\Omega)}^{1/2}$ with the second inequality in \eqref{IneqInterp}, we deduce that
$$
\norm{\nabla \Gv}_{\LL^2(\Gamma_0)} \leq C \norm{\Gv}^{1/4}_{\GLL^2(\Omega)} \norm{\Gv}^{3/4}_{\GHH^2(\Omega)} .
$$
Hence, from \eqref{est1Diffalpai} we obtain:
\BE \label{ProofThmStabNavier}
\norm{\alpha_{1} - \alpha_{2}}_{\LL^2(\MK)}   
\leq \frac{C}{m} \norm{\Gv}^{1/4}_{\GLL^2(\Omega)} \norm{\Gv}^{3/4}_{\GHH^2(\Omega)}
\EE
and we conclude using the estimate on $\norm{\Gv}_{\GLL^2(\Omega)}$ given by Theorem~\ref{ThmStabNS}.

For the Robin boundary conditions, we proceed in exactly the same way to obtain
$$
\norm{\alpha_{1} - \alpha_{2}}_{\LL^2(\MK)}   
\leq \frac{C}{m} \( \norm{\Gv}_{\GLL^2(\Gamma_0)} + \norm{\nabla \Gv}_{\LL^2(\Gamma_0)} + \norm{p}_{\LL^2(\Gamma_0)} \) 
$$
and conclude using the estimate on $\norm{\Gv}_{\GLL^2(\Omega)}$ and on $\norm{p}_{\LL^2(\Omega)}$ given by Theorem~\ref{ThmStabNS}
\begin{remark}\label{EstHk}
We can obtain a better estimate assuming more regularity on $(\Gv,p)$. For example, if $(\Gv,p) \in \GHH^{k}(\Omega) \times \HH^{k-1}(\Omega)$, $k\geq 2$, we can use an interpolation inequality in~\eqref{ProofThmStabNavier} to obtain
\begin{equation}\label{StaInHk}
\norm{\alpha_{1} - \alpha_{2}}_{\LL^2(\mathcal{K})}   
\leq \frac{C}{m} \norm{\Gv}^{1/4}_{\GLL^2(\Omega)} \(\norm{\Gv}^{1-2/k}_{\GLL^2(\Omega)}\norm{\Gv}^{2/k}_{\GHH^k(\Omega)}\)^{3/4} = \frac{C}{m} \norm{\Gv}^{1-3/(2k)}_{\GLL^2(\Omega)} \norm{\Gv}^{3/(2k)}_{\GHH^k(\Omega)}.
\end{equation}
Then \eqref{stabInHk} follows from \eqref{StaInHk} with the interpolation inequality $\norm{\cdot}_{\HH^{\theta n}(\mathcal{K})} \leq C \norm{\cdot}_{\LL^2(\mathcal{K})}^{1-\theta} \norm{\cdot}_{\HH^n(\mathcal{K})}^{\theta}$.
\end{remark}

\begin{remark}
Concerning the Navier boundary conditions, we can obtain the same result in a different way, by writing $[\MD(\Gv)\Gn ]_{\tau} $ in terms of $\curl \Gv$ on $\Gamma_{0}$. Indeed, since $\Gv \cdot \Gn = 0$ on~$\Gamma_{0}$, $\dis \nabla(\Gv \cdot \Gn) = \frac{\partial(\Gv \cdot \Gn)}{\partial\Gn} \Gn$ and then,
\begin{multline*}
\curl \Gv \times \Gn
= (\nabla \Gv - {}^t\nabla \Gv) \Gn 
= \frac{\partial\Gv}{\partial\Gn} - \nabla (\Gv \cdot \Gn) + (\nabla \Gn) \Gv 
\\
= \frac{\partial (\Gv_{\tau} + (\Gv \cdot \Gn)\Gn) }{\partial\Gn} - \frac{\partial(\Gv \cdot \Gn)}{\partial\Gn}\Gn + (\nabla \Gn) \Gv_{\tau}
= \frac{\partial\Gv_{\tau} }{\partial\Gn}  +  (\nabla \Gn) \Gv_{\tau} .
\end{multline*}
On the other hand, using the same kind of computations, we have
\begin{multline*}
2 \MD(\Gv) \Gn = (\nabla \Gv + {}^t \nabla \Gv) \Gn
= \frac{\partial\Gv}{\partial\Gn} + \nabla (\Gv \cdot \Gn) - (\nabla \Gn) \Gv 
= \frac{\partial\Gv_{\tau} }{\partial\Gn}  + 2 \frac{\partial(\Gv \cdot \Gn)}{\partial\Gn}\Gn - (\nabla \Gn) \Gv_{\tau} .
\end{multline*}
Hence, we obtain that 
$$
\curl \Gv \times \Gn = \[ 2 \MD(\Gv) \Gn \]_{\tau} + 2 (\nabla \Gn) \Gv_{\tau} .
$$
Thus, in the previous proof, we can write
$$
(\alpha_{2} - \alpha_{1}) \Gv_{1} = \alpha_{2} \Gv + 2 \nu \[ \MD(\Gv) \Gn \]_{\tau} = \alpha_{2} \Gv + \nu \, \curl \Gv \times \Gn -  2 \nu (\nabla \Gn) \Gv_{\tau}
$$
and conclude using the estimates~\eqref{ineqStabNS} and~\eqref{ineqStabNS02} on $\norm{\Gv}_{\GLL^2(\Omega)}$ and $\norm{\curl \Gv}_{\(\LL^2(\Omega)\)^{2N-3}}$ given by Theorem~\ref{ThmStabNS}.
\end{remark}


\section{Application to error estimates} \label{sectionErrorEsti}

In this section, we consider the reconstruction of $(\Gv ,p)$, solution of the Stokes system in $\Omega$, knowing $\Gv $ and $\sigma(\Gv ,p)  \Gn$ on $\Gamma_{\rm obs}$. In other words, we consider the data completion problem for the Stokes system, that is: from given data $\Gg_D\in \GHH^{3/2}(\Gamma_{\rm obs})$ and $\Gg_N\in \GHH^{1/2}(\Gamma_{\rm obs})$, reconstruct $(\Gv,p)  \in \GHH^2(\Omega)\times \HH^1(\Omega)$ verifying
\BE \label{ErrorEstimStokesPb}
\left\lbrace
\begin{array}{rclcl}
-\nu \Delta \Gv  + \nabla p& =& \Gf & & \mbox{ in } \Omega , \\
 {\rm div}\ \Gv& = &0  & & \mbox{ in } \Omega ,\\
 \Gv& = &\Gg_D  & & \mbox{ on } \Gamma_{\rm obs} , \\
 \sigma(\Gv,p) \Gn &=& \Gg_N  & & \mbox{ on } \Gamma_{\rm obs} .
 \end{array}
\right.
\end{equation}
As the problem is ill-posed, it is mandatory to use  a stabilization method to stably reconstruct $(\Gv,p)$ from the data $\Gf$, $\Gg_D$ and $\Gg_N$. 

Such a stabilization method usually depends on a parameter of regularization $\varepsilon>0$, and it
 must fulfill the two following requirements: it must have a solution for \emph{any data}~$\Gf$,~$\Gg_D$ and~$\Gg_N$, regardless of the existence of a solution to the corresponding Stokes problem~\eqref{ErrorEstimStokesPb}. And its solution should converge to
the solution of~\eqref{ErrorEstimStokesPb} when the parameter $\varepsilon$ goes to zero, when such a solution exists.

We study below two standard methods of regularization: a \emph{quasi-reversibility method}
and a \emph{penalized Kohn-Vogelius method}. In particular, we obtained the convergence rates of these methods directly from 
the estimates obtained previously.

In the following, we denote 
$
\Big( (\Gv,p),(\Gw,q) \Big)_{\GHH^2(\Omega) \times \HH^1(\Omega)} \ov   (\Gv,\Gw)_{\GHH^2(\Omega)} + (p,q)_{\HH^1(\Omega)}
$
which is obviously a scalar product on the Hilbert space $\GHH^2(\Omega) \times \HH^1(\Omega)$, and 
$\Vert (\Gv,p) \Vert_{\GHH^2(\Omega) \times \HH^1(\Omega)}$ the corresponding norm.


\subsection{Error estimates for the quasi-reversibility method}

The quasi-reversibility method has been introduced in \cite{LionsLattes}  by Latt\`es \textit{et al.} to stabilize elliptic, parabolic and hyperbolic ill-posed problem. The main idea of the method is to solve well-posed variational fourth-order problem, depending on~$\varepsilon$.

More precisely, for $\varepsilon>0$, we define the following quasi-reversibility variational problem: find $(\Gv_\varepsilon,p_{\varepsilon}) \in \GHH^2(\Omega) \times \HH^1(\Omega)$ such that $\Gv_\varepsilon = \Gg_D$ on $\Gamma_{\rm obs}$, $\sigma(\Gv_\varepsilon,p_\varepsilon) \Gn = \Gg_N$
on $\Gamma_{\rm obs}$ and for all $(\Gw,q) \in \GHH^2(\Omega) \times \HH^1(\Omega)$ with $\Gw  = \0 $ and $\sigma(\Gw,q)\Gn = \0$ on $\Gamma_{\rm obs}$, we have:
\begin{multline} \label{MethodeQuasiRevers2}
 \int_\Omega  (-\nu \Delta \Gv_\varepsilon + \nabla p_\varepsilon)\cdot (-\nu \Delta \Gw + \nabla q) \, \dd x  
 +  \Big({\rm div}(\Gv_\varepsilon), {\rm div}(\Gw) \Big)_{\HH^1(\Omega)} \\
 + \varepsilon (\Gv_\varepsilon,\Gw)_{\GHH^2(\Omega)} 
  + \varepsilon (p_\varepsilon,q)_{\HH^1(\Omega)} = 
\int_\Omega \Gf \cdot (-\nu \Delta \Gw + \nabla q)\, \dd x.
\end{multline}

We start by proving that the quasi-reversibility problem is well-posed:

\begin{prop} \label{prop_QR}
For any $(\Gf,\Gg_D,\Gg_N) \in \GLL^2(\Omega) \times \GHH^{3/2}(\Gamma_{\rm obs}) \times \GHH^{1/2}(\Gamma_{\rm obs})$, there exists a unique solution $(\Gv_\varepsilon,p_\varepsilon) \in \GHH^2(\Omega) \times \HH^1(\Omega)$ to the quasi-reversibility problem~\eqref{MethodeQuasiRevers2}.
\end{prop}

\begin{proof}
Let us first note that there exists $(\GV,P) \in \GHH^2(\Omega) \times \HH^1(\Omega)$ such that $\GV_{\vert\Gamma_{\rm obs}} = \Gg_D$,
$\sigma(\GV,P) \Gn_{\vert \Gamma_{\rm obs}} = \Gg_N$ and 
$$
\Vert (\GV, P) \Vert_{\GHH^2(\Omega)\times\HH^1(\Omega)} \leq C \Vert (\Gg_D, \Gg_N) \Vert_{\GHH^{3/2}(\Gamma_{\rm obs})\times \GHH^{1/2}(\Gamma_{\rm obs})}.
$$
Indeed, since $\sigma(\GV,P)\Gn\cdot \Gn=2\nu \frac{\partial\GV}{\partial\Gn}\cdot \Gn-P$ and
$$
[\sigma(\GV,P)\Gn]_\tau=\nu\left( \frac{\partial\GV_\tau}{\partial\Gn} +\nabla_\tau (\GV\cdot \Gn)- (\nabla\Gn) \GV_\tau\right)
= \nu\left( \frac{\partial\GV_\tau}{\partial\Gn} + (\nabla_\tau \GV ) \Gn\right),
$$
it suffices to choose $P=0$ and a continuous lifting $\GV$ which satisfies $\GV = \Gg_D$ and
$\nu \frac{\partial\GV}{\partial\Gn} = \frac{1}{2}(\Gg_N\cdot \Gn)\Gn+{\Gg_{N}}_{\tau} - \nu (\nabla_\tau   \Gg_{D}) \Gn$ on $\Gamma_{\rm obs}$.

Defining $(\tilde{\Gv}_\varepsilon   \, , \, \tilde{p}_\varepsilon ) \ov ( \Gv_\varepsilon -  \GV  \, ,  \, p_\varepsilon - P)$, we see that $\tilde{\Gv}_\varepsilon = \0$ and $ \sigma(\tilde{\Gv_\varepsilon},\tilde{p_\varepsilon}) \Gn= \0$ on $\Gamma_{\rm obs}$ and, for all $(\Gw,q) \in \GHH^2(\Omega) \times \HH^1(\Omega)$ such that $\Gw = \0$ and $ \sigma(\Gw,q)\Gn = \0$ on $\Gamma_{\rm obs}$, we have
\begin{multline*}
\int_\Omega  (-\nu \Delta \tilde{\Gv}_\varepsilon + \nabla \tilde{p}_\varepsilon)\cdot (-\nu \Delta \Gw + \nabla q)\, \dd x  + \Big( {\rm div}(\tilde{\Gv}_\varepsilon), {\rm div}(\Gw) \Big)_{\HH^1(\Omega)} \\
 + \varepsilon (\tilde{\Gv}_\varepsilon,\Gw)_{\GHH^2(\Omega)}
 + \varepsilon (\tilde{p}_\varepsilon,q)_{\HH^1(\Omega)}  = 
\int_\Omega \tilde{\Gf} \cdot(-\nu \Delta \Gw + \nabla q)\, \dd x - \int_\Omega  {\rm div}(\GV)\, {\rm div}(\Gw)  \, \dd x \\
 - \varepsilon (\GV,\Gw)_{\GHH^2(\Omega)}
 - \varepsilon (P,q)_{\HH^1(\Omega)} ,
\end{multline*}
where $\tilde{\Gf} \ov  \Gf + \nu \Delta \GV - \nabla P$. The Lax-Milgram theorem 
gives then the result.
\end{proof}

Suppose now that the initial data completion problem admits a (necessarily unique) solution $(\Gv,p) \in \GHH^2(\Omega) \times \HH^1(\Omega)$. 
Then, we have the following

\begin{theorem} \label{thm_QR}
The solution $(\Gv_\varepsilon,p_\varepsilon)\in \GHH^2(\Omega) \times \HH^1(\Omega)$ of the quasi-reversibility problem~\eqref{MethodeQuasiRevers2} converges to $(\Gv,p)\in \GHH^2(\Omega) \times \HH^1(\Omega)$ solution of the data completion problem for the Stokes problem~\eqref{ErrorEstimStokesPb} when $\varepsilon$ tends to zero, strongly in $\GHH^2(\Omega) \times \HH^1(\Omega)$.  We furthermore have the estimate
\begin{equation} \label{QR_4}
\Vert -\nu \Delta \Gv_\varepsilon + \nabla p_\varepsilon - \Gf \Vert^2_{\GLL^2(\Omega)} 
+ \Vert {\rm div}(\Gv_\varepsilon) \Vert^2_{\HH^1(\Omega)} 
\leq \varepsilon \Vert (\Gv,p)\Vert_{\GHH^2(\Omega) \times \HH^1(\Omega)}^2.
\end{equation}
\end{theorem} 

\begin{proof}
Using $(\Gw,q) \ov  (\Gv_\varepsilon - \Gv \,   p_\varepsilon - p)$ as test functions in the quasi-reversibility problem~\eqref{MethodeQuasiRevers2}, which is admissible as they verify
the boundary conditions, we directly obtain
\begin{equation} \label{QR_1}
\Vert -\nu \Delta \Gv_\varepsilon + \nabla p_\varepsilon - \Gf \Vert^2_{\GLL^2(\Omega)} 
+ \Vert {\rm div}(\Gv_\varepsilon) \Vert^2_{\HH^1(\Omega)} + \varepsilon \Big( (\Gv_\varepsilon,p_\varepsilon), (\Gv_\varepsilon - \Gv, p_\varepsilon - p)  \Big)_{\GHH^2(\Omega) \times \HH^1(\Omega)}= 0 .
\end{equation}
 We hence  have $\Big( (\Gv_\varepsilon,p_\varepsilon), (\Gv_\varepsilon - \Gv, p_\varepsilon - p)  \Big)_{\GHH^2(\Omega) \times \HH^1(\Omega)} \leq 0$ which implies
\begin{equation} \label{QR_2}
\Vert (\Gv_\varepsilon,p_\varepsilon) \Vert_{\GHH^2(\Omega) \times \HH^1(\Omega)} \leq 
\Vert (\Gv,p) \Vert_{\GHH^2(\Omega) \times \HH^1(\Omega)}.
\end{equation}
Subtracting $\varepsilon \Big( (\Gv,p), (\Gv_\varepsilon - \Gv, p_\varepsilon - p) \Big)_{\GHH^2(\Omega)\times \HH^1(\Omega)}$ to equation~\eqref{QR_1}, we obtain
\begin{equation} \label{QR_2bis}
\Vert (\Gv_\varepsilon - \Gv,p_\varepsilon - p) \Vert_{\GHH^2(\Omega)\times \HH^1(\Omega)}^2  
\leq  - \Big( (\Gv,p), (\Gv_\varepsilon - \Gv, p_\varepsilon - p) \Big)_{\GHH^2(\Omega)\times \HH^1(\Omega)}
\end{equation}
implying
\begin{equation} \label{QR_3}
\Vert (\Gv_\varepsilon - \Gv,p_\varepsilon - p) \Vert_{\GHH^2(\Omega) \times \HH^1(\Omega)} \leq \Vert (\Gv,p) \Vert_{\GHH^2(\Omega) \times \HH^1(\Omega)}.
\end{equation}
Going back to equation~\eqref{QR_1}, we finally obtain
$$
\Vert -\nu \Delta \Gv_\varepsilon + \nabla p_\varepsilon - \Gf \Vert^2_{\GLL^2(\Omega)} 
+ \Vert {\rm div}(\Gv_\varepsilon) \Vert^2_{\HH^1(\Omega)} \leq \varepsilon \left\vert 
\Big( (\Gv_\varepsilon,p_\varepsilon), (\Gv_\varepsilon - \Gv, p_\varepsilon - p)  \Big)_{\GHH^2(\Omega) \times \HH^1(\Omega)} \right\vert
$$
which, using~\eqref{QR_2} and~\eqref{QR_3}, directly leads to the estimate~\eqref{QR_4}.

Now, suppose that $\Gv_\varepsilon$ and $p_\varepsilon$ do not converge to $\Gv$ and $p$. Then there exist~$\rho >0$ and~$\varepsilon_n$, sequence
of strictly positive real numbers verifying $\varepsilon_n \xrightarrow[n\rightarrow \infty]{} 0$, such that the couple $(\Gv_n  \ov  \Gv_{\varepsilon_n} \, , p_n \ov  p_{\varepsilon_n})$ satisfies 
$$
\Vert \Gv_n - \Gv, p_n - p\Vert_{\GHH^2(\Omega)\times \HH^1(\Omega)} > \rho.
$$
By equation~\eqref{QR_2}, we know that $(\Gv_n,p_n)$ is a bounded sequence in $\GHH^2(\Omega)\times \HH^1(\Omega)$. Hence, up to a subsequence
(that we still denote $(\Gv_n, p_n)$) the sequence converges to some $(\Gw, q)$ weakly in $\GHH^2(\Omega) \times \HH^1(\Omega)$. Then equation~\eqref{QR_4} and the boundary conditions verified by $(\Gv_n,p_n)$ directly imply that $(\Gw,q)$ verifies the Stokes data completion problem~\eqref{ErrorEstimStokesPb}, which in turn implies by uniqueness of such solution that $\Gw = \Gv$ and $q = p$. Therefore, $\Gv_n$ weakly converges to $\Gv$ in $\GHH^2(\Omega)$ and
$p_n$ weakly converges to $p$ in $\HH^1(\Omega)$. But Equation~\eqref{QR_2} implies then that $(\Gv_n,p_n)$ strongly converges
to $(\Gv,p)$, which is a direct contradiction with the definition of the sequence, and therefore ends the proof.
\end{proof}

\begin{remark}
It is not difficult to obtain the following complement to the theorem: if the initial data completion problem for the Stokes system does not admit a solution, then
$$
\Vert (\Gv_\varepsilon ,p_\varepsilon) \Vert_{\GHH^2(\Omega) \times \HH^1(\Omega)} \xrightarrow[\varepsilon \rightarrow 0]{} + \infty.
$$
Otherwise, we would have a sequence of strictly positive real numbers $(\varepsilon_n)_{n\in \mathbb{N}}$ verifying $\varepsilon_n \xrightarrow[n \rightarrow \infty]{} 0$ and $\Vert (\Gv_{\varepsilon_n} ,p_{\varepsilon_n}) \Vert_{\GHH^2(\Omega) \times \HH^1(\Omega)} \leq C $.
But using the same arguments as in the last paragraph of the proof of theorem \ref{thm_QR}, extracting a subsequence and passing to the limit, we would
obtain a solution to the data completion problem for the Stokes system, in obvious contradiction with the assumption.
\end{remark}

Proposition \ref{prop_QR} and Theorem \ref{thm_QR} clearly show that the proposed quasi-reversibility method~\eqref{MethodeQuasiRevers2} is a regularization method for problem~\eqref{ErrorEstimStokesPb}. However, if Theorem~\ref{thm_QR} assures the convergence of the approximated solution to the exact one, it does not give
any rate of convergence. Actually, it is known (see \cite[section~2.5]{Klibanov} and the references therein) that Carleman estimates are the key argument to derive convergence rates for the quasi-reversibility method. 
This is  the case for the quasi-reversibility method proposed above and we now prove Theorem~\ref{ThmRateCv} for this method:

\begin{proof}[Proof of Theorem~\ref{ThmRateCv} for the quasi-reversibility method]
Defining $(\Gu  \, , \, q ) \ov (  \Gv_\varepsilon - \Gv \, , \,  p_\varepsilon - p)$, we notice that we have $ \Gu = \0 $ and  $\sigma(\Gu,q) \Gn = \0$ on $\Gamma_{\rm obs}$ and that the following estimates hold (see Inequalities~\eqref{QR_4} and~\eqref{QR_3}):
\BEAN
 \Vert (\Gu, q) \Vert_{\GHH^2(\Omega) \times \HH^1(\Omega)} 
 &\leq& \Vert (\Gv , p) \Vert_{\GHH^2(\Omega) \times \HH^1(\Omega)} \\
\Vert - \nu \Delta \Gu + \nabla q \Vert_{\GLL^2(\Omega)}
&\leq& \sqrt{\varepsilon} \Vert (\Gv , p) \Vert_{\GHH^2(\Omega) \times \HH^1(\Omega)} \\
 \Vert {\rm div} (\Gu) \Vert_{\LL^2(\Omega)} 
&\leq& \sqrt{\varepsilon} \Vert (\Gv , p) \Vert_{\GHH^2(\Omega) \times \HH^1(\Omega)} . 
\EEAN
Hence, applying estimates~\eqref{ineqStabOseen} and \eqref{ineqStabOseen02}, we directly obtain the result.
\end{proof}

\begin{remark}
Suppose that instead of exact data $(\Gf,\Gg_D,\Gg_N) \in \GLL^2(\Omega_{\rm obs}) \times\GHH^{3/2}(\Gamma_{\rm obs}) \times \GHH^{1/2}(\Gamma_{\rm obs})$, with corresponding solution
$(\Gv, p) \in \GHH^2(\Omega) \times \HH^1(\Omega)$, we have noisy data $(\Gf^\delta,\Gg_D^\delta, \Gg_n^\delta)\in  \GLL^2(\Omega) \times\GHH^{3/2}(\Gamma_{\rm obs}) \times \GHH^{1/2}(\Gamma_{\rm obs})$, such that
$$
\Vert \Gf^\delta - \Gf \Vert_{\GLL^2(\Omega)} \leq \delta,\quad \Vert \Gg_D^\delta - \Gg_D \Vert_{\GHH^{3/2}(\Gamma_{\rm obs})} \leq \delta
\quad \mbox{and} \quad \Vert \Gg_N^\delta - \Gg_N \Vert_{\GHH^{1/2}(\Gamma_{\rm obs})} \leq \delta.
$$
Due to the ill-posedness of the data completion problem~\eqref{ErrorEstimStokesPb}, there might be no solution corresponding to this noisy data. However, the quasi-reversibility
problem~\eqref{MethodeQuasiRevers2} has a corresponding solution, denoted $\Gv_\varepsilon^\delta$ and $p_\varepsilon^\delta$. We also denote $\Gv_\varepsilon$ 
and $p_\varepsilon$ the solution of the quasi-reversibility problem with exact data. It is not difficult to verify that there exists a constant
$C>0$, depending only on the geometry of the domain, such that
$$
\Vert (\Gv_\varepsilon^\delta - \Gv_\varepsilon, p_\varepsilon^\delta - p_\varepsilon) \Vert_{\GHH^2(\Omega) \times \HH^1(\Omega)}
\leq C \frac{\delta}{\sqrt{\varepsilon}}.
$$
Combining this result with the previous estimates, we therefore obtain
$$
 \Vert (\Gv_\varepsilon^\delta  - \Gv , p_\varepsilon^\delta - p) \Vert_{\GHH^1(\Omega) \times \LL^2(\Omega)} 
\leq   C \( \frac{\delta}{\sqrt{\varepsilon}} + \frac{M}{\(\ln(1 + \frac{M}{ \sqrt{\varepsilon} })\)^{1/2}} \) ,
$$
where $M>0$ is such that  $ \|\Gv\|_{{\bf H}^2(\Omega)}+\|p\|_{\HH^{1}(\Omega)} \leq M$. Such estimate highlights the competition between regularization and noise, which leads to the question 
of the optimal choice of the regularization parameter $\varepsilon$ with respect to the amplitude of the noise $\delta$.
On this subject of the optimal choice of the regularization parameter for the quasi-reversibility method for elliptic equations,
see \cite{Bourgeois,Cao} and the references therein. 
\end{remark}


\subsection{Error estimates for the Kohn-Vogelius method}

The quasi-reversibility method proposed in the previous section regularizes the data completion problem for the Stokes system by solving approximately the  first two equations  of~\eqref{ErrorEstimStokesPb} (see
the estimate in Theorem~\ref{thm_QR}) while verifying exactly the boundary conditions. The Kohn-Vogelius method we study now is somehow a symmetric method, in the sense that it solves exactly the equations in $\Omega$ with approximated boundary conditions. 
And again, we obtain the rate of convergence of the method using the same estimates~\eqref{ineqStabOseen} and~\eqref{ineqStabOseen02}.

We recall that $\Gamma_{\rm obs}^C \ov  \partial \Omega \setminus \overline{\Gamma_{\rm obs}}$ and that we here assume that $\overline{\Gamma_{\rm obs}} \cap \overline{\Gamma_{\rm obs}^C} = \emptyset$. For $\Gvarphi_N \in \GHH^{1/2} (\Gamma_{\rm obs}^C)$ and  $\Gpsi_D \in \GHH^{3/2} (\Gamma_{\rm obs}^C)$, we denote $(\Gv_{\Gvarphi_N},p_{\Gvarphi_N}) \in \GHH^2(\Omega) \times \HH^1(\Omega)$
and $(\Gv_{\Gpsi_D},p_{\psi_D}) \in \GHH^2(\Omega) \times \HH^1(\Omega)$ the respective solutions of 
\begin{equation} \label{PbStokesKV}
\left\lbrace
\begin{array}{rcll}
\!\!-\nu \Delta \Gv_{\Gvarphi_N} + \nabla p_{\varphi_N}& =& \Gf & \! \mbox{ in } \Omega , \\
 {\rm div}\ \Gv_{\Gvarphi_N}& = &0 &\! \mbox{ in } \Omega , \\
 \Gv_{\Gvarphi_N} & = &\Gg_{D} &\! \mbox{ on } \Gamma_{\rm obs} , \\
 \sigma(\Gv_{\Gvarphi_N},p_{\varphi_N}) \Gn &=& \Gvarphi_N &\! \mbox{ on } \Gamma_{\rm obs}^C ,
 \end{array}
\right.
\mbox{and }
\left\lbrace
\begin{array}{rcll}
\!\!-\nu \Delta \Gv_{\Gpsi_D} + \nabla p_{\psi_D} & =& \Gf &\! \mbox{ in } \Omega , \\
 {\rm div}\ \Gv_{\Gpsi_D}& = &0 &\! \mbox{ in } \Omega , \\
 \sigma(\Gv_{\Gpsi_D},p_{\psi_D}) \Gn &=& \Gg_{N}  & \!\mbox{ on } \Gamma_{\rm obs} , \\
 \Gv_{\Gpsi_D} & = &\Gpsi_D  & \! \mbox{ on } \Gamma_{\rm obs}^C .
 \end{array}
\right. 
\end{equation}
We define the non-negative functional 
$$
F \, : \,  
(\Gvarphi_N,\Gpsi_D) \in \GHH^{1/2}(\Gamma_{\rm obs}^C) \times \GHH^{3/2}(\Gamma_{\rm obs}^C) 
\mapsto
\vert \Gv_{\Gvarphi_N} - \Gv_{\Gpsi_D} \vert_{\GHH^2(\Omega)}^2 
+ \vert \Gv_{\Gvarphi_N} - \Gv_{\Gpsi_D} \vert_{\GHH^1(\Omega)}^2 
\in \Rr,
$$
 where $\vert \cdot \vert_{\GHH^1(\Omega)} \ov \norm{\nabla(\cdot)}_{\GLL^2(\Omega)}$ and $\vert \cdot \vert_{\GHH^2(\Omega)} \ov \norm{\nabla^2(\cdot)}_{\GLL^2(\Omega)}$ are the respective $\GHH^1$ and~$\GHH^2$-seminorm.

\begin{remark}
For this Kohn-Vogelius method, we have to impose $\overline{\Gamma_{\rm obs}} \cap \overline{\Gamma_{\rm obs}^C} = \emptyset$ in order to guarantee that the functional is well-defined and more precisely that the pairs $(\Gv_{\Gvarphi_N},p_{\Gvarphi_N}) $ and $(\Gv_{\Gpsi_D},p_{\psi_D})$ belong to $\GHH^2(\Omega) \times \HH^1(\Omega)$ for all $(\Gvarphi_N,\Gpsi_D) \in \GHH^{1/2}(\Gamma_{\rm obs}^C) \times \GHH^{3/2}(\Gamma_{\rm obs}^C) $ . 
In the case  $\overline{\Gamma_{\rm obs}} \cap \overline{\Gamma_{\rm obs}^C} \neq \emptyset$, one cannot guarantee that the solutions belong to $\GHH^2(\Omega) \times \HH^1(\Omega)$ (see for example~\cite{Sav97}).
\end{remark}

It is not difficult to verify that the two following propositions are equivalent:
\begin{itemize}
\item there exists $(\Gvarphi_N,\Gpsi_D) \in \GHH^{1/2}(\Gamma_{\rm obs}^C) \times \GHH^{3/2}(\Gamma_{\rm obs}^C)$ such that $F(\Gvarphi_N,\Gpsi_D) = 0$; 
\item there exists a (necessarily unique) solution to the data completion problem~\eqref{ErrorEstimStokesPb}. 
\end{itemize}
Hence one could try to reconstruct the solution of problem~\eqref{ErrorEstimStokesPb} by minimizing $F$. However, this is not a stable strategy:
indeed, the infimum of $F$ is always $0$ even if \eqref{ErrorEstimStokesPb} does not admit a solution, but in this case there are minimizing sequences $(\Gvarphi_{N}^m ,\Gpsi_{D}^m )$ such that
\begin{equation}\label{MinseqInfty}
\lim_{m \rightarrow \infty} \Vert (\Gvarphi_{N}^m , \Gpsi_{D}^m)  \Vert_{\GHH^{1/2}(\Gamma_{\rm obs}^C) \times \GHH^{3/2}(\Gamma_{\rm obs}^C)} = + \infty.
\end{equation}
Let us briefly explain why. Due to the denseness of the admissible data
(see~\cite{Belgacem} or \cite[section~2]{Darde}), for all $\varepsilon>0$, there exists $\( \Gf_{\varepsilon} , \Gg_{D_{\varepsilon}}, \Gg_{N_{\varepsilon}} \) \in \GLL^2(\Omega) \times \GHH^{3/2}(\Gamma_{\rm obs}) \times \GHH^{1/2}(\Gamma_{\rm obs})$ such that the corresponding Stokes problem~\eqref{ErrorEstimStokesPb} has a solution $(\Gv_{\varepsilon},p_{\varepsilon})$ and
$$
\norm{\Gf_{\varepsilon} - \Gf} \leq \varepsilon \, , \quad \norm{\Gg_{D_{\varepsilon}} - \Gg_{D}} \leq \varepsilon \; \mbox{ and } \quad \norm{\Gg_{N_{\varepsilon}} - \Gg_{N}} \leq \varepsilon  .
$$
Therefore, choosing $\Gvarphi_{N_{\varepsilon}} = \sigma(\Gv_{\varepsilon},p_{\varepsilon}) \Gn$ and $\Gpsi_{D_{\varepsilon}} = \Gv_{\varepsilon}$ on $\Gamma_{\rm obs}^C$, it is not difficult to see that 
\begin{equation} \label{MinF0}
 0 \leq F(\Gvarphi_{N_{\varepsilon}} , \Gpsi_{D_{\varepsilon}}) \leq C \varepsilon .
\end{equation}
 Hence, the infimum of $F$ is $0$. Furthermore, if the above sequence $(\Gvarphi_{N_{\varepsilon}} , \Gpsi_{D_{\varepsilon}})$ is bounded in $\GHH^{1/2}(\Gamma_{\rm obs}^C) \times \GHH^{3/2}(\Gamma_{\rm obs}^C)$, one can extract a weakly convergent subsequence which leads to the existence of a solution of the Cauchy problem~\eqref{ErrorEstimStokesPb} using Problems~\eqref{PbStokesKV} and Inequality~\eqref{MinF0}.

Thus, to regularize the problem, we add a penalization term: for $\varepsilon>0$, we introduce the functional $F_\varepsilon \, : \,  \GHH^{1/2}(\Gamma_{\rm obs}^C) \times \GHH^{3/2}(\Gamma_{\rm obs}^C) \rightarrow \Rr$ defined by
$$
F_\varepsilon  (\Gvarphi_N,\Gpsi_D) = 
 F(\Gvarphi_N,\Gpsi_D) + \varepsilon \Vert (\Gv_{\Gvarphi_N},p_{\Gvarphi_N}) \Vert_{\GHH^2(\Omega) \times \HH^1(\Omega)}^2 + \varepsilon \Vert (\Gv_{\Gpsi_D},p_{\Gpsi_D}) \Vert_{\GHH^2(\Omega) \times \HH^1(\Omega)}^2 .
$$
We have the following result:

\begin{prop} \label{prop_KV}
For any $(\Gf,\Gg_D,\Gg_N) \in \GLL^2(\Omega) \times \GHH^{3/2}(\Gamma_{\rm obs}) \times \GHH^{1/2}(\Gamma_{\rm obs})$, there exists a unique 
$(\Gvarphi_{N}^\varepsilon  ,\Gpsi_{D}^\varepsilon  ) \in \GHH^{1/2}(\Gamma_{\rm obs}^C) \times \GHH^{3/2}(\Gamma_{\rm obs}^C)$ such that 
$$
F_\varepsilon(\Gvarphi_{N}^\varepsilon  ,\Gpsi_{D}^\varepsilon  )  = \min_{(\Gvarphi_N,\Gpsi_D) \in \GHH^{1/2}(\Gamma_{\rm obs}^C) \times \GHH^{3/2}(\Gamma_{\rm obs}^C)} F_\varepsilon(\Gvarphi_N,\Gpsi_D) .
$$
\end{prop}

\begin{proof}
Obviously, the functional $F_\varepsilon$ is continuous and strictly convex. 
Furthermore, it is coercive. Indeed, suppose it is not. Then there exists
a sequence $(\Gvarphi_{N}^m ,\Gpsi_{D}^m )$ and a constant $C>0$ such that
\BEAN
\lim_{m \rightarrow \infty}  \Vert (\Gvarphi_{N}^m , \Gpsi_{D}^m)  \Vert_{\GHH^{1/2}(\Gamma_{\rm obs}^C) \times \GHH^{3/2}(\Gamma_{\rm obs}^C)} = + \infty
& \text{ and } &
F_\varepsilon( \Gvarphi_{N}^m , \Gpsi_{D}^m ) <C.
\EEAN
This directly implies $ \Vert ( \Gv_{\Gvarphi_{N}^m },p_{\Gvarphi_{N}^m }) \Vert_{\GHH^2(\Omega) \times \HH^1(\Omega)} <C $ and
$
\Vert (\Gv_{\Gpsi_{D}^m },p_{\Gpsi_{D}^m }) \Vert_{\GHH^2(\Omega) \times \HH^1(\Omega)}<C
$, which directly implies $ \Vert (\Gvarphi_{N}^m , \Gpsi_{D}^m )\Vert_{\GHH^{1/2}(\Gamma_{\rm obs}^C) \times \GHH^{3/2}(\Gamma_{\rm obs}^C)} < C$ by continuity of
trace and normal derivative operators, which is a contradiction with the initial assumptions.

Therefore $F_\varepsilon$ is continuous, strictly convex and coercive, which implies the result (see~\cite{Ekeland}).
 \end{proof}

Suppose now that the initial data completion problem admits a (necessarily unique) solution $(\Gv,p) \in \GHH^2(\Omega) \times \HH^1(\Omega)$. 
Then, we have the following

\begin{theorem}
The solution $(\Gv_{\Gvarphi_{N}^\varepsilon },p_{\Gpsi_{D}^\varepsilon})\in \GHH^2(\Omega) \times \HH^1(\Omega)$ converges to $(\Gv,p)\in \GHH^2(\Omega) \times \HH^1(\Omega)$ solution of the data completion problem for the Stokes problem~\eqref{ErrorEstimStokesPb} when $\varepsilon$ tends to zero, strongly in $\GHH^2(\Omega) \times \HH^1(\Omega)$.
\end{theorem}

\begin{proof}
We denote $\Gvarphi_{N}^{\rm ex} \ov  \sigma(\Gv,p) \Gn_{\vert \Gamma_{\rm obs}^C}$ and $\Gpsi_{D}^{\rm ex} \ov  \Gv_{\vert \Gamma_{\rm obs}^C}$. By definition, we have
$(\Gv_{\Gvarphi_{N}^{\rm ex}} ,p_{\Gvarphi_{N}^{\rm ex}})  = (\Gv_{\Gpsi_{D}^{\rm ex}}, p_{\Gpsi_{D}^{\rm ex}}) = (\Gv,p) $ and 
$F(\Gvarphi_{N}^{\rm ex}, \Gpsi_{D}^{\rm ex})  = 0$. Therefore, by definition of~$\Gvarphi_{N}^\varepsilon  $ and~$\Gpsi_{D}^\varepsilon  $, we have
\begin{multline*}
\vert \Gv_{\Gvarphi_{N}^\varepsilon  } - \Gv_{\Gpsi_{D}^\varepsilon  } \vert_{\GHH^2(\Omega)}^2 
+\vert \Gv_{\Gvarphi_{N}^\varepsilon  } - \Gv_{\Gpsi_{D}^\varepsilon  } \vert_{\GHH^1(\Omega)}^2 
+\varepsilon \Vert (\Gv_{\Gvarphi_{N}^\varepsilon  },p_{\Gvarphi_{N}^\varepsilon  }) \Vert_{\GHH^2(\Omega) \times \HH^1(\Omega)}^2  + \varepsilon \Vert (\Gv_{\Gpsi_{D}^\varepsilon  },p_{\Gpsi_{D}^\varepsilon  }) \Vert_{\GHH^2(\Omega) \times \HH^1(\Omega)}^2 \\
 \leq F_\varepsilon(\Gvarphi_{N}^{\rm ex}, \Gpsi_{D}^{\rm ex}) =2 \varepsilon \Vert (\Gv,p) \Vert_{\GHH^2(\Omega) \times \HH^1(\Omega)}^2
\end{multline*}
which directly implies
\begin{equation} \label{proof_conv_KV_1}
\vert \Gv_{\Gvarphi_{N}^\varepsilon  } - \Gv_{\Gpsi_{D}^\varepsilon  } \vert_{\GHH^2(\Omega)}^2 
+\vert \Gv_{\Gvarphi_{N}^\varepsilon  } - \Gv_{\Gpsi_{D}^\varepsilon  } \vert_{\GHH^1(\Omega)}^2 
\leq 2 \varepsilon \Vert (\Gv,p) \Vert_{\GHH^2(\Omega) \times \HH^1(\Omega)}^2
\end{equation}
and
\begin{equation} \label{proof_conv_KV_2}
 \Vert (\Gv_{\Gvarphi_{N}^\varepsilon  },p_{\Gvarphi_{N}^\varepsilon  }) \Vert_{\GHH^2(\Omega) \times \HH^1(\Omega)}^2
 +
 \Vert (\Gv_{\Gpsi_{D}^\varepsilon  },p_{\Gpsi_{D}^\varepsilon  }) \Vert_{\GHH^2(\Omega) \times \HH^1(\Omega)}^2
 \leq 2  \Vert (\Gv,p) \Vert_{\GHH^2(\Omega) \times \HH^1(\Omega)}^2.
\end{equation}

Let us consider now an arbitrary sequence of positive real numbers $\varepsilon_m$ such that $\displaystyle \lim_{m \rightarrow \infty} \varepsilon_{m} = 0$.
From~\eqref{proof_conv_KV_2}, we see that 
$$
(\Gv_{\Gvarphi_{N}^m },p_{\Gvarphi_{N}^m }) \ov  (\Gv_{\Gvarphi_{N_{\varepsilon_m}}},p_{\Gvarphi_{N_{\varepsilon_m}}})
\quad \mbox{ and } \quad 
(\Gv_{\Gpsi_{D}^m },p_{\Gpsi_{D}^m }) \ov  (\Gv_{\Gpsi_{D_{\varepsilon_m}}},p_{\Gpsi_{D_{\varepsilon_m}}})
$$
are bounded in $\GHH^2(\Omega) \times \HH^1(\Omega)$. Therefore, up to a subsequence, we have the following weak convergences in $\GHH^2(\Omega)$
$$
\Gv_{\Gpsi_{D}^m } \rightharpoonup \Gv_{\Gpsi_{D_\infty}}, \quad \Gv_{\Gvarphi_{N}^m } \rightharpoonup \Gv_{\Gvarphi_{N_\infty}}
$$
and the following weak convergences in $\HH^1(\Omega)$
$$
p_{\Gpsi_{D}^m } \rightharpoonup p_{\Gpsi_{D_\infty}}, \quad p_{\Gvarphi_{N}^m } \rightharpoonup p_{\Gvarphi_{N_\infty}}.
$$
But Equation~\eqref{proof_conv_KV_1} implies directly that $\Gv_{\Gpsi_{D_\infty}} =\Gv_{\Gvarphi_{N_\infty}}+ \Gc $, with $\Gc \in \mathbb{R}^N$, and passing to the limit in the first equations in each Stokes problem of~\eqref{PbStokesKV}, we get $p_{\psi_{D_{\infty}}}=p_{\varphi_{N_{\infty}}}+c$, with $c\in \Rr$. 
In particular, passing to the limit gives $\Gv_{\Gvarphi_{N_\infty}}  = \Gg_{D}$ and $ 
\sigma(\Gv_{\Gvarphi_{N_\infty}}, p_{\Gpsi_{D_\infty}}) \Gn 
= \sigma(\Gv_{\Gpsi_{D_\infty}}, p_{\Gpsi_{D_\infty}}) \Gn   = \Gg_N$ on $\Gamma_{\rm obs}$ by weak continuity of the trace and normal derivative on~$\Gamma_{\rm obs}$. Therefore $( \Gv_{\Gvarphi_{N_\infty}} , p_{\Gpsi_{D_\infty}} ) =\( \Gv , p\)$. Hence, we have the following weak convergences in $\GHH^2(\Omega)$
$$
\Gv_{\Gpsi_{D}^m } \rightharpoonup \Gv + \Gc, 
\quad
 \Gv_{\Gvarphi_{N}^m } \rightharpoonup \Gv
$$
and the following weak convergences in $\HH^1(\Omega)$
$$
p_{\Gpsi_{D}^m } \rightharpoonup p, 
\quad 
p_{\Gvarphi_{N}^m } \rightharpoonup p + c.
$$

Now, we see that $F(\Gvarphi_{N}^{\rm ex},\Gpsi_{D}^{\rm ex} ) = 0 = F(\Gvarphi_{N}^{\rm ex}+ c\, \Gn,\Gpsi_{D}^{\rm ex}+\Gc)$ for any $c \in \mathbb{R}$ and $\Gc \in \mathbb{R}^N$. Therefore, similarly has previously,
we have
$
F_\varepsilon(\Gvarphi_{N}^m , \Gpsi_{D}^m ) \leq F_\varepsilon(\Gvarphi_{N}^{\rm ex}+c\, \Gn,\Gpsi_{D}^{\rm ex}+\Gc)$
which implies 
\begin{multline*}
 \Vert (\Gv_{\Gvarphi_{N}^\varepsilon  },p_{\Gvarphi_{N}^\varepsilon  }) \Vert_{\GHH^2(\Omega) \times \HH^1(\Omega)}^2
  +
 \Vert (\Gv_{\Gpsi_{D}^\varepsilon  },p_{\Gpsi_{D}^\varepsilon  }) \Vert_{\GHH^2(\Omega) \times \HH^1(\Omega)}^2
   \leq\Vert (\Gv,p + c)\Vert_{\GHH^2(\Omega) \times \HH^1(\Omega)}^2 \\ 
    + 
 \Vert (\Gv+\Gc,p) \Vert_{\GHH^2(\Omega) \times \HH^1(\Omega)}^2
\end{multline*}
directly implying that the weak convergences are actually strong convergences.

Finally, a standard argument \textit{ad absurdum} ends the proof as in the end of the proof of Theorem~\ref{thm_QR}.
\end{proof}

We now prove Theorem~\ref{ThmRateCv} for this penalized Kohn-Vogelius method, recalling that $(\Gv_\varepsilon,p_{\varepsilon}) \ov  (\Gv_{\Gvarphi_{N}^\varepsilon  }  , p_{\Gpsi_{D}^\varepsilon  })$.

\begin{proof}[Proof of Theorem~\ref{ThmRateCv} for the penalized Kohn-Vogelius method]
It is not difficult to verify that we have the \textit{a priori} bounds (see~\eqref{proof_conv_KV_2})
$$
\Vert \Gv_\varepsilon - \Gv \Vert_{\GHH^2(\Omega)} \leq C(\Gv,p),\quad  \Vert p_\varepsilon - p\Vert_{\HH^1(\Omega)}
\leq \tilde{C}(\Gv,p)
$$
where $C(\Gv,p)$ and $\tilde{C}(\Gv,p)$ are constants depending only on the $\GHH^2(\Omega)\times \HH^1(\Omega)$ norm of~$(\Gv,p)$. Furthermore, thanks to~\eqref{proof_conv_KV_1}, we see that
\begin{eqnarray} \label{SemiNormH2}
 \Vert \sigma(\Gv_\varepsilon,p_\varepsilon) \Gn - \Gg_N \Vert_{\GHH^{1/2}(\Gamma_{\rm obs})} 
& = & \Vert \sigma(\Gv_{\Gvarphi_{N}^\varepsilon  },p_{\Gpsi_{D}^\varepsilon  }) \Gn - \sigma(\Gv_{\Gpsi_{D}^\varepsilon  },p_{\Gpsi_{D}^\varepsilon  }) \Gn  \Vert_{\GHH^{1/2}(\Gamma_{\rm obs})}  \nonumber \\
  &  \leq & \vert \Gv_{\varepsilon} - \Gv_{\Gpsi_{D}^\varepsilon  } \vert_{\GHH^2(\Omega)}+\vert \Gv_{\varepsilon} - \Gv_{\Gpsi_{D}^\varepsilon  } \vert_{\GHH^1(\Omega)} \\ 
   &  \leq &  \sqrt{\varepsilon}\,  C(\Gv,p) \nonumber 
\end{eqnarray}
where $C(\Gv,p)$ is another constant depending only on $\GHH^2(\Omega)\times \HH^1(\Omega)$ norm of~$(\Gv,p)$.

Hence, applying again estimates~\eqref{ineqStabOseen} and~\eqref{ineqStabOseen02}, we directly obtain the announced result.
\end{proof}

\begin{remark}
The Kohn-Vogelius functional is classically defined by $\MF (\Gvarphi_N,\Gpsi_D) = \vert \Gv_{\Gvarphi_N} - \Gv_{\Gpsi_D} \vert_{\GHH^1(\Omega)}^2 $ instead of $F (\Gvarphi_N,\Gpsi_D) = \vert \Gv_{\Gvarphi_N} - \Gv_{\Gpsi_D} \vert_{\GHH^2(\Omega)}^2  + \vert \Gv_{\Gvarphi_N} - \Gv_{\Gpsi_D} \vert_{\GHH^1(\Omega)}^2 $. Notice that Proposition~\ref{prop_KV} is also valid for the associated functional $\MF_{\varepsilon}$. The only point where the $\GHH^2$-seminorm is needed is Inequality~\eqref{SemiNormH2}.
\end{remark}

\bibliographystyle{abbrv}
\bibliography{biblio2}

\end{document}